%% file: euclidean_splitting_arxiv.tex
\documentclass[titlepage]{amsart}

\input{macros_kmx.tex}

\usepackage{amsaddr}

\usepackage{verbatim}

\usepackage{tikz}

\definecolor{fg}{rgb}{0.136,0.556,0.136}
\usepackage{xfrac}

\usepackage[makeroom]{cancel}
\newcommand{\g}{\nabla }

\title[Sobolev mappings and product structure]
{ Sobolev mappings of Euclidean space and product structure}

\author[B. Kleiner]{Bruce Kleiner}
\thanks{BK was supported by NSF grants DMS-1405899 and DMS-1711556 and a Simons Collaboration grant.}
\address{Courant Institute of Mathematical Sciences, New York University, USA}

\author[S. M\"uller]{Stefan M\"uller}
\thanks{SM has been supported by the Deutsche Forschungsgemeinschaft (DFG, German Research Foundation) through
the Hausdorff Center for Mathematics (GZ EXC 59 and 2047/1, Projekt-ID 390685813) and the 
collaborative research centre  {\em The mathematics of emerging effects} (CRC 1060, Projekt-ID 211504053).  This work was initiated during a sabbatical of SM at the Courant Institute and SM would like to thank  R.V. Kohn and the Courant Institute
members and staff for 
their  hospitality and a very inspiring atmosphere.}
\address{Institute for Applied Mathematics, University of Bonn, Germany}

\author[L. Sz{\'{e}}kelyhidi]{ L\'{a}szl\'{o} Sz{\'{e}}kelyhidi, Jr.}
\thanks{   LSz gratefully acknowledges the support of the Deutsche Forschungsgemeinschaft (DFG, German Research Foundation) through GZ SZ 325/2-1.}
\address{Max Planck Institute for Mathematics in the Sciences, Leipzig, Germany}

\author[X. Xie]{Xiangdong Xie}
\thanks{XX has been supported by Simons Foundation grant \#315130.}
\address{Department of Mathematics and Statistics, Bowling Green State University, USA}

\begin{document}

\begin{abstract}  
 We consider bounded open connected sets $\Omega_1, \Omega_2 \subset \R^n$ and 
 Sobolev maps $f: \Omega_1 \times \Omega_2 \subset \R^n \times \R^n$ such that for almost every $x \in \Omega_1 \times \Omega_2$
the weak differential $\g f(x)$ is invertible  and  preserves or swaps the spaces $\R^n \times \{0\}$ and $\{0\} \times \R^n$.  
We show that if 
 $n \ge 2$ and $f \in W^{1,2}$  
then $f$ is split, i.e., $f(x_1, x_2) = (f_1(x_1), f_2(x_2))$
or $f(x_1, x_2) = (f_2(x_2), f_1(x_1))$. 

We also show that this conclusion fails in general for $n=1$, even if we assume in addition that $f$ is bi-Lipschitz
and area preserving. 
These results complement the work \cite{kmsx_counterexample}, where we showed that the conclusion fails for $n \ge 2$ if the Sobolev space $W^{1,2}$ is replaced by $W^{1,p}$
for any $p < 2$. 

We also discuss results for approximately split maps, i.e.~for  sequences of  maps $f_k$
such that $\g f_k$ approaches the set of linear invertible split maps in suitable $L^p$ spaces.

This work is partly motivated by the question whether Sobolev maps defined on products of Carnot groups are split, see \cite{KMX1new}.
\end{abstract}

\maketitle

\setcounter{tocdepth}{2}
\tableofcontents

\section{Introduction}~

If $X_1,X_2\subset \R^n$ are subsets, we say that a mapping  $f:X_1\times X_2\ra \R^{2n}$ is split (or preserves product structure) if there exist functions $f_1 : X_1 \to  \R^n$ and $f_2 : X_2 \to \R^n$
  such that either $f(x,y) = (f_1(x),f_2(y))$ for all $(x,y) \in X_1\times X_2$
  or $ f(x,y)=(f_2(y), f_1(x)) $ for all   $(x,y) \in X_1\times X_2$. As in our previous work  \cite{kmsx_counterexample}, we are interested in the following question about mappings $f:\Om_1\times\Om_2\ra \R^{2n}$, where $\Om_1,\Om_2\subset \R^n$ are connected open subsets and $f$ is assumed to be either Lipschitz, bi-Lipschitz, or in $W^{1,p}_{\loc}$ for some $1\leq p<\infty$.

\begin{question}  \label{qu:global_split}  
If the (approximate) differential $\g f(x)$ is  split and invertible for almost every $x\in \Om$, is $f$ split?  More generally, if the differential is ``approximately split'', must $f$ itself be ``approximately split''?
\end{question}
\noindent
Our motivation for considering this question comes   from   geometric group theory, geometric mapping theory, and the theory of nonlinear partial differential equations; see the end of the introduction for discussion of this context. 

From now on we fix two connected open subsets $\Om_1,\Om_2\subset \R^n$, and let $\Om:=\Om_1\times\Om_2$. 

Note that Question~\ref{qu:global_split} is trivial for $C^1$ maps: if $f:\Om\ra \R^{2n}$ is $C^1$  and the differential $\g f(x)$ is bijective and split everywhere, then $f$ is clearly split, 
since $\g f:\Om\ra \R^{2n\times 2n}$ is a continuous map taking values in the set of split and bijective linear maps, which consists of two components -- the block diagonal and the block anti-diagonal invertible matrices. On the other hand, if $f:\Om\ra \R^{2n}$ is Lipschitz then its differential is only measurable, so in principle oscillations between the two types of behavior might arise.  In fact, for $n = 1$ it is easy to find Lipschitz maps such that $\g f$  is bijective and split a.e., but $f$ is not split. For instance consider the `folding map'
$$f(x_1,x_2)=  \frac12  \binom{x+y + h(x-y)}{x+y - h(x-y)}$$
where $h : \R\to  \R$ is a Lipschitz function with $h' =  \pm 1$ a.e.
 (for a specific example one may take $h(t) = |t|$). 
 Then  
$$
\g f(x) = \begin{pmatrix} 1 & 0 \\ 0 & 1 \end{pmatrix} \quad \text{or}
\quad 
\g f(x) = \begin{pmatrix} 0 & 1 \\ 1 & 0 \end{pmatrix} 
$$
for almost every $x\in \R^2$ but  $f$ is not split unless $h' \equiv 1$ or $h' \equiv -1$. 
This example reflects the fact that for $n = 1$ the set of split, bijective linear maps $\R^2\ra \R^2$ contains rank-one connections
 between the diagonal and antidiagonal matrices; that is, there exists a diagonal matrix and an antidiagonal matrix whose difference has rank one.  When $n\geq 2$, no such rank-one connections exist between the block diagonal and block antidiagonal invertible matrices, so no analogs of the folding examples exist, and hence one might expect a positive answer to Question~\ref{qu:global_split} for Lipschitz maps. 
For a similar reason one might expect a positive answer to Question~\ref{qu:global_split} for a bi-Lipschitz mapping $f:\Om\ra\R^2$: the sign of $\det \g f(x)$ is constant almost everywhere (because it agrees with the local degree of $f$) and there are no rank-one connections between diagonal and antidiagonal matrices whose determinant has the same sign. 

\subsection{Our results}

Our first result confirms the expectation of rigidity in the $n\geq 2$ case, even for Sobolev mappings, as announced in \cite{kmsx_counterexample}:
\begin{theorem}  \label{th:split}
 Suppose $n\geq 2$ and  $f  \in  W^{1,2}_{\loc}(\Omega;\R^n)$. If the weak differential $\g f(x)$  is split and bijective for a.e.~$x \in \Omega$, then $f$ is split.
 \end{theorem}
The Sobolev exponent $2$ is sharp:  for every $p < 2$ there exists a $W^{1,p}_{\loc}$-mapping $f:\Om\ra\R^{2n}$ such that $\g f(x)$ is split and  $\det \g f(x) = 1$ for almost every $x\in \Om$,
yet $f$ is not split, see  \cite[Theorem 1.2]{kmsx_counterexample}.  This exhibits the subtle dependence of rigidity/flexibility on the \emph{a priori} regularity assumptions, established for other nonlinear PDEs     \cite{nash54, scheffer74, tartar79, murat81, gromov_pdr, scheffer93,muller_sverak96,dacorogna_marcellini99, muller99,muller_sverak03,Kirchheim:2002wc,astala_faraco_szekelyhidi08,delellis_szekelyhidi09,CDS2012,SzekelyhidiJr:2014tu,delellis_szekelyhidi16,isett18,MR4944781,buckmaster_vicol19}.  

The assumption that $\g f$ is bijective almost everywhere cannot be dropped.
Consider, for example, the Lipschitz  map given by $f_1(x) = x_1 + x_{n+1} + |x_1 - x_{n+1}|$, $f_2 = \ldots = f_{2n} = 0$. This map  satisfies
$\g f \in \{ 2 e_1 \otimes e_1, 2 e_1 \otimes e_{n+1}\}$ a.e., but is not globally split.

\bigskip

For the remainder of the introduction, we assume in addition that our connected open subsets $\Om_1,\Om_2\subset \R^n$ are bounded.

In the $n=1$ case, Question~\ref{qu:global_split} turns out to have a negative answer even for bilipshitz homeomorphisms.

\begin{theorem}  \label{th:nosplit} If  $n = 1$, then 
 there exists a bi-Lipschitz homeomorphism $f : \Om  \to  \R \times  \R$ such that:
\begin{enumerate}[label=(\alph*)]
\item $\g f(x)$ 
is split and bijective for a.e.~$x$;
\item There is a null set $N$ such that $\g f(x)$ takes only five values for $x\not\in N$;
\item  $f$ is area preserving: $\det \g f = 1$ for a.e.~$x$;
\item $f$ agrees with a non-split  affine map on $\partial \Omega$; in particular $f$ is not split.
\end{enumerate}
\end{theorem}

Theorem~\ref{th:nosplit} implies:
\begin{corollary}
\label{cor_global_nosplit}
 There exists a non-split bi-Lipschitz homeomorphism $f:\R\times\R\ra\R\times\R$ satisfying assertions (a)-(c) of Theorem~\ref{th:nosplit}.
\end{corollary}

\medskip
We now consider the stability of the rigidity assertion in the $n\geq 2$ case,  establishing a quantitative no-oscillation result for approximately split maps.  Denote by $L\subset \R^{2n\times 2n}$ the set of split matrices and let $L_1$ and $L_2$ be the subsets
of matrices which map $\R^{n} \times \{0\}$ to itself or to $\{0\} \times \R^n$, respectively. In other words, in $n\times n$ block-matrix form we have $L=L_1\cup L_2$ with 
\begin{equation*}
    L_1=\left\{ \begin{pmatrix} A&0\\0&D\end{pmatrix}:\, A,D\in \R^{n\times n}\right\},\quad  L_2=\left\{ \begin{pmatrix} 0&B\\C&0\end{pmatrix}:\, B,C\in \R^{n\times n}\right\}.
\end{equation*}

Let $f_j:\Om\ra \R^{2n}$ be a sequence of maps  which is bounded in $W^{1, 2n}$.  We show that  if $\g f_j$ converges to $L$
 and  $\det \g f_j$ is controlled from below, 
then  $\g f_j$ converges to $L_1$ or to $L_2$.  In particular any   weak limit  $f$ is  split. 
Throughout this paper
$$ \text{ we use the half-arrow $\rightharpoonup$ to denote weak convergence.}$$
\begin{theorem}    \label{th:approximate_solutions} 
Suppose that $n \ge  2$ and
\begin{eqnarray}  \label{eq:weak_convergence_fj}
f_j &\rightharpoonup &f    \quad \text{in $W^{1, 2n}(\Omega, \R^{2n})$},\\
 \label{eq:convergence_to_split}
\dist( \g f_j, L) &\to& 0 \quad \text{in $L^{1}(\Omega)$},
\end{eqnarray}
and
\begin{equation}  \label{eq:lower_bound_det}
\lim_{\delta \downarrow 0} \limsup_{j \to \infty}  | \{ x \in \Omega : \det \g f_j(x) < \delta \}| = 0.
\end{equation}
Then  $\g f \in L$ a.e. and hence  $f$ is globally split. 
Morevoer
\begin{equation}  \label{eq:approx_strong_to_Li}
\dist( \g f_j, L_i) \to 0 \quad \text{in $L^{q}(\Omega)$  \quad for $i=1$ or for  $i=2$.}
\end{equation}
and all $q < 2n$. 
\end{theorem}

\begin{remark}  \label{re:almost_split}~
\ben \item Weak convergence in $W^{1,2n}$ cannot be replaced by weak convergence in $W^{1,p}$ for any $p < 2n$,
even if we replace  \eqref{eq:convergence_to_split} and  \eqref{eq:lower_bound_det} by the stronger conditions
$\dist(\g u_j, L) \to 0$ in $L^s(\Omega)$ for all $s < \infty$ and 
$|\{ \det \g u_j \ne 1\}| \to 0$,
see \cite[Theorem 1.4]{kmsx_counterexample}.
\item  Condition \eqref{eq:lower_bound_det} is in particular satisfied if $\det \g f_j \ge \delta > 0$ almost everywhere. 
Condition \eqref{eq:lower_bound_det}  cannot be replaced by either of the two conditions $\det \g f_j > 0$ or $|\det \g f_j| \ge \delta >0$,  see Examples~\ref{ex:counterexample_det_bigger_zero} and \ref{ex:det_pm_approximate}. 
\item  The proof shows that  for the conclusion that $\g f$ is (globally) split, 
it actually suffices to assume $f_j \rightharpoonup f$ (i.e.~weakly) in $W^{1, 2}(\Omega, \R^{2n})$, $\dist(\g f_j, L) \to 0$ in $L^1(\Omega)$
and $\det \g f \ne 0$ a.e. 
Without weak convergence in $W^{1,2n}$ one can, however, 
 in general, not get information on $\det \g f$ from $\det \g f_j$.
\item \label{it:almost_split_vs_split}
  If $\Omega_1$ and $\Omega_2$ have Lipschitz boundary, it follows from  the compact Sobolev embedding that the $f_j$ are close to the  split map $f$  in
$L^{q}$ for all       $q < \infty$.  One might wonder whether there exist split maps $g_j$ which are close to $f_j$ in $W^{1,1}$.
The following example shows that this is in general not the case. Let $n=2$,  $\Omega_i = (0,1)^2$, let $h \in C^1(\R)$ be $1$-periodic,  
and let $\varphi \in C^1([0,1])$. Consider the maps $f_j:\Omega\to\R^4$ 
\begin{equation*}
f_j(x)=\left(x_1+\tfrac1j h(jx_2) \varphi(x_3),x_2,x_3,x_4\right).
\end{equation*}
Then $\det \g f_j = 1$ and $\dist(\g f_j, L) \le C j^{-1}$.
If $g$ is a globally split map, then $\partial_2 g_1$ is independent of $x_3$.
Using the estimate 
$$
\| \varphi - \bar \varphi\|_{L^1(0,1)}  \le 2  \|\varphi - c \|_{L^1(0,1)}  \quad \forall c \in \R$$
where  
$$\bar \varphi = \int_0^1 \varphi(t) \, dt, $$
we see that 
\begin{equation*}
\begin{split}
  \| \g f_j - \g g\|_{L^1((0,1)^4)} &\ge \|h'\varphi-\partial_2g_1\|_{L^1((0,1)^4)}\\
  &\ge \frac12\|h'\|_{L^1(0,1)}\|\varphi - \bar \varphi\|_{L^1(0,1)}
  \end{split}
\end{equation*}
for all  split maps $g$.
\een
\end{remark}

\bigskip

\begin{remark}
We remark in passing that Theorem \ref{th:approximate_solutions} for approximately split maps can be stated  and proved very concisely in the language of  gradient 
Young measures.  These measures capture the one-point statistics of a sequence of gradients. More precisely, a map  $\nu$ from $\Omega$ to the set  $\mathcal P(\R^{2n \times 2n})$ of probability measures on $\R^{2n \times 2n}$
is a $W^{1,p}$ gradient Young measure if there exists a sequence of $W^{1,p}$ maps $f_j: \Omega \subset \R^{2n} \to \R^{2n}$ such
that  $|\g f_j|^p \rightharpoonup g$ in $L^1(\Omega)$ and,
$$  \psi \circ \g f_j \rightharpoonup \bar \psi \quad 
\text{in $L^1(\Omega)$ \quad  with  \quad $\bar \psi(x) = \int_{\R^{2n \times 2n}} \psi(X) \, d\nu_x(X)$}
$$
for a.e. $x \in \Omega$ and 
 for every continuous function $\psi: \R^{2n \times 2n} \to \R$ which satisfies
 $$ \psi(X) \le C (1 + |X|^p)$$
 for some $C > 0$.

We say that a gradient Young measure  $\nu$ is supported in a Borel set $A \subset \R^{2n \times 2n}$ if $\nu(x)(\R^{2n \times 2n} \setminus  A) = 0$ for a.e.\ $x \in \Omega$.
A $W^{1,p}$ gradient Young measure $\nu$ is called homogeneous if there exists a probability measure $\mu$ such that $\nu_x = \mu$ for a.e.
$x \in \Omega$. In this case, by abuse of notation, one also calls $\mu$ a homogeneous $W^{1,p}$ gradient Young measure. With this preparation we can restate Theorem \ref{th:approximate_solutions} as
\begin{theorem}   \label{th:hgym} 
Let $n \ge 2$ and let $\Sigma_+ = \{ X \in \R^{2n \times 2n} : \det X > 0 \}$.
If $\Omega \subset \R^{2n}$ is bounded, open, and connected and $\nu: \Omega \to \mathcal P(\R^{2n \times 2n})$ is a $W^{1, 2n}$
   gradient Young measure which is supported in $L \cap \Sigma_+$,
then $\nu$ is supported in $L_1  \cap \Sigma_+$
 or in $L_2  \cap \Sigma_+$.
\end{theorem}
In concurrence with Remark \ref{re:almost_split}(1) we can restate the sharpness of the exponent $2n$ as follows.
Let $\Sigma_1 = \{ X  \in \R^{2n \times 2n}: \det X = 1\} \subset \Sigma_+$. Then there exists a probability measure $\mu$ which is supported in $L \cap \Sigma_1$, such that $\mu$ is a   $W^{1,p}$ gradient Young measure for all $p < 2n$, and satisfies
$$ \mu(L_1) > 0 \quad \text{and} \quad \mu(L_2) > 0.$$
\end{remark}

\bigskip

\subsection{Context}
Question~\ref{qu:global_split} originated from rigidity questions, which arose in geometric group theory and geometric mapping theory.  We give a brief indication of this connection here, describing only the simplest case; for more details and context see \cite{KMX1new}. Let $\H$ denote the Heisenberg group equipped with the Carnot-Carath\'eodory metric and the usual bi-invariant  measure. Recall that $\H$ has topological dimension $3$ and homogeneous dimension $4$; in particular the volume of a metric ball of radius $r$ is given by $c r^4$.  The simplest question about products is the following:

\begin{question}
\label{que_heisenberg_product}
 If $f:\H\times\H\ra \H\times\H$ is a bi-Lipschitz homeomorphism, must $f$ be split?  
 Here $\H \times \H$ is equipped with the product metric and product measure.
\end{question}

\noindent
The map $f$ is Pansu differentiable almost everywhere \cite{pansu};  by definition, the Pansu differential $D_P f(x)$
is a (graded) group automorphism of $\H \times \H$,  and it is a little exercise in linear algebra to show that $D_P f(x)$ either preserves the first and the second factor or swaps them.  Thus the Pansu differential is split, i.e. it preserves product structure, and Question~\ref{que_heisenberg_product} reduces to a problem formally identical to Question~\ref{qu:global_split}, except that $\R^n$ is replaced by the Heisenberg group, and the usual differential is replaced by the Pansu differential.
It was shown in  \cite{KMX1new} that $D_P f$ cannot oscillate between these two behaviours and hence $f$ is split. 
This assertion also holds for Sobolev mappings: if $f:\H\times \H\ra \H\times \H$ is a $W^{1,p}_{\loc}$ 
-mapping for $p\geq 3$, and the approximate Pansu differential $D_Pf(x)$ is invertible almost everywhere, then $f$ is split \cite{kmx_approximation_low_p}.  It is not known whether this conclusion also holds for $p < 3$.

\smallskip

We have indicated above how Question~\ref{qu:global_split} arose from a rigidity question in the setting of Carnot groups.  It turns out that 
our discussion of Question~\ref{qu:global_split} also yields mappings between  Carnot groups;  these are of  interest in connection with rigidity of Iwasawa groups, see Remark~\ref{rem_f1_f2_rigidity} and \cite{kmx_iwasawa}. 

Let $\fh$ denote the Lie algebra of $\H$ with standard basis $X_1,X_2,X_3$, and let $\fh=V_1\oplus V_2$ be the grading, where $V_1=\Span\{X_1,X_2\}$, $V_2=\Span\{X_3\}$.  We identify $V_1$ with $\R^2$ by $X_i\leftrightarrow e_i\in \R^2$.   Given a Lipschitz map $f:\H\ra\H$, we may also view the Pansu differential as a graded Lie algebra homomorphism $Df(x):\fh\ra \fh$; restricting to the horizontal subspace $V_1\subset\fh$, we obtain the {\bf horizontal differential} $d_Hf(x):=Df(x)\restr_{V_1}:V_1\ra V_1$.   Combining Corollary~\ref{cor_global_nosplit} with a lifting argument  yields bi-Lipschitz mappings of the Heisenberg group whose horizontal differential splits, but have oscillatory behavior.

\begin{corollary}
\label{cor_heisenberg_example} 
There is a bi-Lipschitz homeomorphism $\hat f:\H\ra \H$  such that for a.e. $x\in\H$, the horizontal differential $d_H\hat f(x):\R\times\R\simeq V_1\ra V_1\simeq \R\times \R$ is split, but $\hat f$ does not preserve the left coset foliations for the $1$-parameter subgroups generated by $X_1$ and $X_2$ (i.e. $d_H\hat f$ exhibits oscillatory behavior). 
\end{corollary}

\bigskip

\subsection{Organisation} In Section~\ref{se:split} we  prove the results for $n \ge 2$. 
In Section~\ref{se:nosplit}  we construct a non-split map using the theory of convex integration.
We show that there exists five  split $2 \times 2$  matrices $X_1, \ldots, X_5$ with determinant one,
a non-split $2 \times 2$ matrix $A$ and a Lipschitz map $f: \Omega \to \Omega$ such that
$\g f \in \{X_1, \ldots, X_5\}$ a.e. and $f(x) = Ax$ on $\partial \Omega$. 
By a result of F\"orster and the third author \cite{szekelyhidi_forster18},  such maps exist provided
that the five matrices $X_1, \ldots, X_5$ form a so-called large $T_5$ configuration. 
In Appendix~\ref{sec_proof_heisenberg_corollary} we provide more details on the context in the Heisenberg setting and give the proof of Corollary \ref{cor_heisenberg_example}.

\section{Proof of  splitting for $n \ge 2$.}  \label{se:split}

In this section we prove the results for $n \ge 2$: Theorem~\ref{th:split},  Theorem~\ref{th:approximate_solutions}, 
and Theorem~\ref{th:hgym}.

\subsection{Split maps}
The proof  of 
Theorem~\ref{th:split} is based on the
fact that minors (subdeterminants) of the gradient of a map $f : U \subset
\R^m \to \R^d$ satisfy certain compatiblity relations. For example, the
$1\times 1$ minors of the differential of  a $C^2$ map  satisfy 
$ \frac{\partial}{\partial x_m} \frac{\partial f_i}{\partial x_j} =  \frac{\partial}{\partial x_j} \frac{\partial f_i}{\partial x_m}$.
For higher order
minors the compatibility conditions can be very efficiently encoded in the language of differential forms. 
Recall that for a $k$-form $\alpha  = \sum_{i_1, \ldots, i_k} a_{i_1 \ldots i_k}(y) \, dy_{i_1} \wedge \ldots \wedge dy_{i_k}$
on $\R^d$ the pullback by $f$ is defined as the following $k$-form on $U$
$$
(f^* \alpha)(x) = \sum_{i_1, \ldots, i_k} a_{i_1 \ldots i_k}(f(x))  \sum_{j_1, \ldots, j_k} 
\frac{\partial f_{i_1}}{\partial x_{j_1}} dx_{j_1} \wedge \ldots \wedge \frac{\partial f_{i_k}}{\partial x_{j_k}} dx_{j_k}.
$$
Note that by antisymmetry of the wedge product the right hand side depends only on the $k \times  k$ minors of $\g f$. 
The compatibility condition on the minors is expressed by the fact that pullback commutes with exterior differentiation. Specifically,
 we use the following result.
 
\begin{lemma}  \label{le:pullback_closed}
 Let $U \subset  \R^m$, let $\alpha$  be a smooth $k$-form on $\R^d$ and let $f \in W^{1,k}(U;\R^d)$. If $\alpha$  is closed, then $f^* \alpha$ is weakly closed, i.e., for every smooth $m -k-1$-form
  $\beta$ which is compactly supported in $U$, we have
  \begin{equation}
  \int_U f^* \alpha \wedge d\beta = 0.
  \end{equation}
 More generally, if $\alpha$ is a general smooth $k$-form on $\R^d$ and $f \in W^{1,k+1}(U;\R^d)$
  then (weak)  exterior differentiation and pullback commute, i.e.,
  \begin{equation}
  \int_U f^* \alpha \wedge d\beta = (-1)^{k+1} \int_U  f^*d\alpha \wedge \beta.
  \end{equation}
  \end{lemma}

\begin{proof} If $f$ is smooth and $\alpha$  is a smooth $k$-form on $\R^d$ then 
we have $d f^* \alpha = f^* d \alpha$.  Indeed, this follows easily by induction, 
starting with $k = 1$ and using the identities $f^*(\alpha \wedge \gamma) = f^* \alpha \wedge f^* \gamma$
 and     $d(\alpha \wedge \gamma) = d\alpha \wedge \gamma +  (-1)^k  \alpha \wedge d\gamma$ for a $k$-form
 $\alpha$ and an $\ell$-form $\gamma$.  
 Thus for maps $f$  which are smooth on the support of $\beta$   we have by Stokes'  theorem

\begin{align*}&  \, \int_U f^*\alpha \wedge d\beta  \\= & (-1)^k \int_U d(f^*\alpha \wedge \beta)  - (-1)^k \int df^*\alpha \wedge \beta \\
=  & \,  0 +  (-1)^{k+1} \int f^*d\alpha \wedge \beta 
\end{align*}
Now  $f^*\alpha$  depends only on the $k \times k$  minors of $\g f$ and $f^*d\alpha$ depends only on the $(k+1) \times (k+1)$ minors of $\g f$.   Thus the assertions follow,
 since a $W^{1,p}$ map can be approximated by $C^\infty$  maps in $W^{1,p}_{\rm loc}$.
 \end{proof}

\begin{proof}[Proof of Theorem~\ref{th:split}]
Recall that $L \subset \R^{2n \times 2n}$ denotes the set of split matrices, $L_1 \subset L$
is the set of split matrices which preserve $\R^n \times \{0\}$,  and $L_2 \subset L$ is the 
set of split matrices which map $\R^n  \times \{0\}$ to $\{0\} \times \R^n$. 
Since $\g f \ne 0$ a.e., there exists  a measurable function 
$\chi: \Omega \to \{0, 1\}$ such 
that 
\begin{equation} \label{eq:define_chi_new}
\chi(x) = \begin{cases} 1 & \text{if $\g f(x) \in L_1$,} \\
0 & \text{if $\g f(x) \in L_2$.}
\end{cases}
\end{equation}
We claim that 
\begin{equation} \label{eq:claim_chi_new}
\chi = 0 \quad \text{almost everywhere} \quad \text{or} \quad \chi = 1 \quad \text{almost everywhere.}
\end{equation}
From this claim one easily deduces that $f$ is split. 

The pullback of the form   $\alpha = dy_1 \wedge dy_2$   
   is given by
\begin{align*} f^*\alpha =  df_1 \wedge df_2 = \sum_{1 \le  i < j \le 2n}  M_{ij}(\g f) dx_i \wedge dx_j
\end{align*}
where 
$$ M_{ij} = \det \begin{pmatrix} \partial_i f_1 & \partial_j f_1 \\
\partial_i f_2 & \partial_j f_2
\end{pmatrix}.
$$
Since $\g f(x) \in L$ a.e. the terms with $i \le n$ and $j \ge n+1$ vanish a.e. Thus
\begin{equation}
f^*\alpha = \sum_{1 \le i < j \le n} a_{ij} dx_i \wedge dx_j + \sum_{n+1 \le i < j \le 2n} b_{ij} dx_i \wedge dx_j
\end{equation}
where 
\begin{equation}  \label{eq:chi_determines_a_b}  \sum_{i < j} a_{ij}^2 = 0 \quad \text{if $\chi = 0$} \quad \text{and} \quad  \sum_{i < j} b_{ij}^2 = 0
\quad \text{if $\chi = 1$}.
\end{equation}
We now claim that 
\begin{eqnarray}
\frac{\partial a_{ij}}{\partial x_l} &=& 0 \quad \text{if $n+1 \le l \le 2n$,}   \label{eq:distributional_aij}\\
\frac{\partial b_{ij}}{\partial x_l} &=& 0 \quad \text{if $1 \le l \le n$,} \label{eq:distributional_bij}
\end{eqnarray}
in the sense of distributions. 
To show the result for $i=1, j=2$ and $l = 2n$,  we apply Lemma~\ref{le:pullback_closed} with $k=2$ and
$$ \beta = \varphi \omega  \quad \text{where} \quad \omega = dx_3 \wedge \ldots \wedge dx_{2n-1}$$
and $\varphi \in C_c^\infty(\Omega)$. 
Then 
$$ d \beta =  \sum_{m \in \{1,2, 2n\}} \frac{\partial \varphi}{\partial x_m} dx_m \wedge \omega.$$
To compute $f^*\alpha \wedge d\beta$ we first note that 
$$ dx_i \wedge dx_j \wedge dx_m \wedge \omega = 0 \quad \text{if $n+1 \le i < j \le 2n$} $$
since this form contains $n+1$ terms $dx_l$ with $l \in \{n+1, \ldots, 2n\}$.
Similarly 
$$ dx_i \wedge dx_j \wedge dx_m \wedge \omega = 0 \quad \text{if $1 \le i < j \le n$ and $m \in \{1, 2\}$.} $$
Thus 
$$ f^*\alpha \wedge d\beta =  f^*\alpha \wedge \frac{\partial \varphi}{\partial x_{2n}} dx_{2n} \wedge \omega = 
- \frac{\partial \varphi}{\partial x_{2n}}    a_{12}  \,   dx_1 \wedge \ldots \wedge x_{2n}$$
and Lemma~\ref{le:pullback_closed} yields
$$ \int_\Omega \frac{\partial \varphi}{\partial x_{2n}} a_{12}  \, dx = 0 \quad \forall \varphi \in C_c^\infty(\Omega). $$
This shows that  \eqref{eq:distributional_aij} holds for $i=1$, $j=2$ and $k= 2n$.
The remaining assertions follow in the same way by taking $\omega$ as  the $(2n-3)$-form $dx_{l_1} \wedge \ldots \wedge dx_{l_{2n-3}}$
where $dx_i$, $dx_j$ and $dx_k$ are omitted. The proof of \eqref{eq:distributional_bij} is analogous.

It is easy to see that \eqref{eq:chi_determines_a_b} ,  \eqref{eq:distributional_aij}, and  \eqref{eq:distributional_bij} imply the assertion.
We include the details for the convenience of the reader.
Recall that $\Omega = \Omega_1 \times \Omega_2$ with $\Omega_i \subset \R^n$ open and connected. 
It follows from \eqref{eq:distributional_aij} and \eqref{eq:distributional_bij} that there exist measurable functions
$A_{ij}: \Omega_1 \to \R$ and $B_{ij}: \Omega_2 \to \R$ such that
$$ a_{ij} = A_{ij} \quad \text{and} \quad b_{ij} = B_{ij} \quad \text{almost everywhere}.$$
Let 
$$E_1 = \{ x' \in \Omega_1 : \sum_{i < j} A_{ij}^2(x') \ne 0 \},  \quad E_2 = \{ x'' \in \Omega_2 : \sum_{i < j} B_{ij}^2(x'') \ne 0 \}.$$
Then  \eqref{eq:chi_determines_a_b} implies that
$$ \chi = 1 \quad \text{a.e. on $E_1 \times \Omega_2$},  \quad  \chi = 0 \quad \text{a.e. on $\Omega_1 \times E_2$}.$$
Hence $E_1 \times E_2 =   (E_1 \times \Omega_2) \cap (\Omega_1 \times E_2)$ is a null set. It follows that  $E_1$ or $E_2$ is a null set.
If $E_1$ is a null set,  then
$$ f^* \alpha = \sum_{n+1 \le i < j \le 2n} b_{ij} dx_i \wedge dx_j  \quad \text{almost everywhere.}$$
In particular 
$$ 
f^* \alpha = 0 \quad  \text{almost everywhere in $\Omega_1 \times (\Omega_2 \setminus E_2)$.}
$$
Thus $\rank \g f \le  2n-1$ a.e. in $\Omega_1 \times (\Omega_2 \setminus E_2)$. Since by assumption $\g f$ is bijective a.e.,
the set $\Omega_2 \setminus E_2$ must be a null set. Hence $\sum_{i < j} b_{ij}^2 \ne 0$ a.e.,  and in view of 
\eqref{eq:chi_determines_a_b} this implies
that $\chi = 0$ a.e. 

If $E_2$ is a null set, then we show similarly that $\chi =1$ a.e.
This concludes the proof of \eqref{eq:claim_chi_new}.
\end{proof}

\bigskip
\subsection{Approximately split maps} \label{se:approximately_split}
We now turn to the proof of Theorem~\ref{th:approximate_solutions}.
We use the following result on the weak continuity of subdeterminants (or minors)

\begin{lemma}   \label{le:weak_continuity} Let $\Omega \subset \R^m$ be open and bounded. Let $M(F)$ be a $k \times k$ subdeterminant
of the $d \times m$ matrix $F$. Assume that 
$$ f_j \rightharpoonup f \quad \text{in $W^{1,k}(\Omega; \R^d)$}. $$
Then
$$ M(\g f_j) \overset{*}{\rightharpoonup} M(\g f) \quad \text{in measures,} $$
i.e,
$$ \int_\Omega M(\g f_j) \, \varphi \, dx \to  \int_\Omega M(\g f) \, \varphi \, dx   \quad \forall \varphi \in C_c(\Omega).$$
\end{lemma}
Equivalently, for every smooth $k$-form $\omega$ the sequence of pullbacks $f_j^* \omega$ converges weak* in measures
to $f^*\omega$. 

\begin{proof} 
This follows easily by induction over $k$ from the   fact that minors of order $k$ arise from the pullback of  the forms $dy^{i_1} \wedge \cdots \wedge dy^{i_k} = d(y^{i_1} \wedge dy^{i_2} \wedge \cdots \wedge dy^{i_k})$ and Lemma~\ref{le:pullback_closed}. For a proof which does not use differential forms,   see 
\cite[Theorem 8.20]{dacorogna_book08}.
\end{proof}

We also collect  some simple facts about minors which will be useful in the proof. 
If $I =(i_1, \ldots, i_r)$ with $1 \le i_1 < \ldots < i_r \le 2n$  and $J = (j_1, \ldots, j_r)$ with
$1 \le j_1 < \ldots < j_r \le 2n$  and  $F \in \R^{2n \times 2n}$ we denote by $F_{IJ}$ the submatrix
with rows $i_1, \ldots, i_r$ and column $j_1, \ldots, j_r$ and we set
$M_{IJ}(F) = \det F_{IJ}$. 

\begin{lemma}   \label{le:minor_facts}
\ben  
\item  \label{it:le_minor_expansion}
 Let $F \in \R^{2n \times 2n}$ and let $F' \in L$ be such that 
$|F- F'| = \dist(F, L)$. If $M$ is an $r \times r$ minor then 
\begin{equation}
|M(F) - M(F')| \le c (|F|^{r-1} \dist(F, L) + \dist^r(F, L));
\end{equation}
\item   \label{it:lemma_minor_converge}
If $M$ is an $r \times r$ minor which vanishes on $L$, $p \ge r$ and  $F_j : \Omega \to \R^{2n \times 2n}$ satisfies
 $$ \sup_j \|  F_j\|_{L^p} < \infty \quad \text{and}  \quad  \dist(F_j, L) \to 0 \quad \text{in $L^p(\Omega)$} $$
  then
\begin{equation} M(F_j) \to 0 \quad \text{in $L^{p/r}(\Omega)$;}
\end{equation}
\item  \label{it:lemma_minor_off_diagonal}
Assume that  $F$ has the block-diagonal form 
\begin{equation} \label{eq:lemma_minor_block}
F =  \begin{pmatrix} A & B \\ C & D \end{pmatrix} \quad \text{with $A, B, C, D \in \R^{n \times n}$}
\end{equation}
and $\dist(F, L_2) < \dist (F, L_1)$. Then 
\begin{align} 
& \, | \det F - (-1)^n \det B \det C| \\
 \le & \,  c' (  |F|^{2n-1} \, \dist(F, L) + \dist^{2n}(F,L)). \nonumber 
\end{align}
\item  \label{it:lemma_minor_two_minors} Consider the $n \times 2n$ matrix
$G = \begin{pmatrix} A & B \end{pmatrix}$.  Let 
\begin{align} \label{eq:define_Lprime_minor}
 L' = \{ &G :  M_{IJ}(G) = 0, \, \text{whenever $I = (i_1, i_2)$ with  $1 \le i_1 < i_2 \le n$} \\
 &\text{and $J = (j_1, j_2)$ with  $1 \le j_1 \le n < j_2 \le 2n$} \}. \nonumber
 \end{align}
 Then  $ G \in L'$ if and only if 
 \begin{equation} \label{eq:vanishing_two_minors_on_L}
 \rank G = 1 \quad \text{or} \quad A = 0 \quad \text{or} \quad B=0.
 \end{equation}
\item \label{it:lemma_minor_characterization} Assume $n\geq 2$. Let $F$ be as in \eqref{eq:lemma_minor_block} with $\begin{pmatrix} A & B \end{pmatrix}\in L'$ and $\begin{pmatrix} C & D \end{pmatrix}\in L'$ and $\det F\neq 0$. Then $F\in L$. In particular, if $F\in \R^{2n\times 2n}$ is nonsingular and $M(F)=0$ for all $2\times 2$ minors $M$ vanishing on $L$, then $F\in L$.
\een
\end{lemma}

\begin{proof} \eqref{it:le_minor_expansion}: This follows from the fact that $M$ is a homogeneous polynomial of degree $r$ and Young's inequality.

\eqref{it:lemma_minor_converge}: This follows directly from \eqref{it:le_minor_expansion}.

 \eqref{it:lemma_minor_off_diagonal}:  There exists $F' \in L_2$ such that 
 $$ |F- F'| = \dist(F, L_2) = \dist(F, L).$$
 If we write $F'$ in block diagonal form with block matrices $0, B', C', 0$
 then $\det F' = (-1)^n \det B' \det C'$. Now
 $$ |B - B'| + |C - C'| \le c |F-F'| = c \dist(F, L)$$
 for some constant $c$. Thus the assertion follows by applying    \eqref{it:le_minor_expansion} to $\det$ and
 to the minors which correspond to the determinant of the submatrices $B$ and $C$. 
 
  \eqref{it:lemma_minor_two_minors}:  If the condition  \eqref{eq:vanishing_two_minors_on_L} holds, then clearly $G \in L'$.
  For the converse implication note that
the assumptions and the conclusion are invariant under multiplication of $G$ by non-singular $n \times n$ matrices
on the left and by non-singular block-diagonal $2n \times 2n$ matrices on the right. 
Thus, if $A \ne 0$,  we may assume that $A$ is diagonal with entries $1$ or $0$ on the diagonal, i.e.,
 $A = \sum_{i=1}^r  e_i \otimes e_i$. Using the minors with $i_1=1$ and $j_1=1$ we see  that  $B_{jl'} = 0$ if $j \ge 2$
 and $l' = l-n \ge 1$.   If $r = 1$ then it follows that $\rank G = 1$. If $r \ge 2$ then we can also use the minors with
 $i_1=2$ and $j_2=2$ and we deduce that $B= 0$. 

 \eqref{it:lemma_minor_characterization}: First of all, from \eqref{it:lemma_minor_two_minors} we deduce 
$$ \rank \begin{pmatrix} A & B \end{pmatrix} = 1 \quad \text{or} \quad A = 0 \quad \text{or} \quad B=0$$
and similarly
$$ 
\rank \begin{pmatrix} C & D \end{pmatrix} = 1 \quad \text{or} \quad C = 0 \quad \text{or} \quad D=0.
$$
Further, since $\det F \ne 0$, we have $ \rank \begin{pmatrix} A & B \end{pmatrix} = n$ and
$ \rank \begin{pmatrix} C& D \end{pmatrix} = n$. Moreover, we cannot have $A=C=0$ or $B=D=0$.
Thus,  we necessarily have $A=D=0$ or $B=C=0$. Hence $F \in L$.
  \end{proof}

\begin{proof}[Proof of Theorem~\ref{th:approximate_solutions}]
In both arguments, the  key observation is that along sequences which satisfy $\dist(\g f_j, L) \to 0$ in $L^{2n}$ there is an   additional
weakly continuous expression which agrees with $\det$ on $L_2$ and vanishes on $L_1$, see Step 2 of the proof and  Lemma~\ref{le:cc}  below.

We first show the assertions of the theorem under the additional assumption
\begin{equation}  \label{eq:approx_equiintegrable}
\text{the sequence  $|\g f_j|^{2n}$ is equiintegrable.} 
\end{equation}
Recall that  a  sequence of $L^1$ functions $h_j$ is equiintegrable if for every $\eps > 0$ there exists a
$\delta > 0$ such that $|A| < \delta$ implies $\sup_j \int_A |h_j| < \eps$. The Dunford-Pettis theorem shows that
if $\Omega$ has finite measure  and if the sequence  $h_j: \Omega \to \R^s$ is equiintegrable then $h_j$ has a subsequence which converges
weakly  in $L^1(\Omega)$. 

{\bf Step 1:} \emph{If the additional assumption  \eqref{eq:approx_equiintegrable} holds, then $f$ is split.}\\
We first claim that $\det \g f > 0$ a.e.
  It suffices to show that 
 \begin{equation}  \label{eq:criterion_det_positive}
 \int_U \det \g f \, dx > 0 \quad \text{for all measurable $U \subset \Omega$ with $|U| > 0$.}
 \end{equation}

Fix a measurable $U \subset \Omega$ with $|U| > 0$. 
Lemma~\ref{le:weak_continuity} yields the convergence $\det \g f_j \overset{*}{\rightharpoonup} \det \g f$ weak*
in measures.
In view of  \eqref{eq:approx_equiintegrable} and the Dunford-Pettis theorem we get
\begin{equation}  \label{eq:approx_weakL1_det}
\det \g f_j   \rightharpoonup \det \g f \quad \text{in $L^1(\Omega)$}
\end{equation}
and in particular 
\begin{equation}  \label{eq:weak_L1_U}
 \lim_{j \to \infty} \int_U \det \g f_j \, dx  = \int_U \det \g f \, dx.
\end{equation}
Set $E_{j, \delta} = \{ x \in \Omega : \det \g f_j < \delta \}$.
By assumption $\lim_{\delta \downarrow 0} \limsup_{j \to \infty} |E_{j, \delta}|= 0$.
Hence the equi-integrability of the sequence  $\det \g f_j$ implies that
\begin{equation}   \label{eq:weak_det_bad}
\lim_{\delta \downarrow 0} \liminf_{j \to \infty}  \int_{U \cap E_{j,\delta}} \det \g f_j \, dx = 0.
\end{equation}
Moreover, there exists a $\delta_0 > 0$ such that   $\limsup_{j \to \infty} |E_{j, \delta_0}| \le |U|/2$.
Thus 
\begin{equation*}  \label{eq:weak_det_good}
\lim_{\delta \downarrow 0}   \liminf_{j \to \infty}  \int_{U \setminus E_{j,\delta}} \det \g f_j\, dx
\ge     \liminf_{j \to \infty}  \int_{U \setminus E_{j,\delta_0}} \det \g f_j\, dx
 \ge  \frac12  \delta_0 |U|.
\end{equation*}
Adding  
  \eqref{eq:weak_det_bad}  to this inequality and using 
\eqref{eq:weak_L1_U} we see that $\int_U \det \g f \, dx > 0$.

We now show $\g f \in L$ a.e. Then Theorem~\ref{th:split} implies that $f$ is 
split.
By decomposing  $\Omega$ into the sets $\{ |\g f_j| \le M\}$ and $\{ |\g f_j| > M\}$ we easily see that   the assumption $\dist(\g f_j, L) \to 0$ in $L^1(\Omega)$ and equiintegrability of $|\g f_j|^{2n}$ imply that 
\begin{equation}  \label{eq:convergence_dist_L2n}
\dist(\g f_j, L) \to 0 \quad \text{in $L^{2n}(\Omega)$}.
\end{equation}

The matrix $\g f(x)$ has the block decomposition
\begin{equation*}
\g f(x) = \begin{pmatrix} A & B \\ C & D \end{pmatrix} \quad \text{with $A,B,C,D \in \R^{n \times n}$.}
\end{equation*}
It follows from 
 \eqref{eq:convergence_dist_L2n} and 
Lemma~\ref{le:minor_facts}~\eqref{it:lemma_minor_converge}
that  $M(\g f_j) \to 0$ in $L^1(\Omega)$  for all $2 \times 2$ minors which vanish on
$L$. Together with the weak continuity of minors we deduce that
$$ 
M(\g f) = 0 \quad \text{for all $2 \times 2$ minors which vanish on $L$.}
$$
Since also $\det \g f \ne 0$ a.e., we deduce from Lemma \ref{le:minor_facts} \eqref{it:lemma_minor_characterization} that $\nabla f \in L$ a.e. Thus,  by 
Theorem~\ref{th:split}, $f$ is  split. In particular, either $\nabla f \in L_1$ a.e. or $\nabla f \in L_2$ a.e.

\medskip

  {\bf Step 2:} \emph{If the additional assumption  \eqref{eq:approx_equiintegrable} holds, then $\dist(\g f_j, L_i) \to 0$ 
in $L^{2n}(\Omega)$  for $i=1$ or $i=2$.} \\ 
Assume for definiteness that $\g f \in L_1$ a.e
(the case $\g f \in L_2$ a.e. is analogous).
 Write
\begin{equation*}
\g f_j(x) = \begin{pmatrix} A_j & B_j \\ C_j & D_j \end{pmatrix} \quad \text{with $A_j,B_j,C_j,D_j \in \R^{n \times n}$.}
\end{equation*}
The key observation is that
 \eqref{eq:convergence_dist_L2n} implies that
\begin{equation}  \label{eq:weak_continuity_detBC}
\det B_j \det C_j \overset{*}{\rightharpoonup}  \det B \det C = 0 \quad \text{in $\mathcal M(\Omega)$,}
\end{equation}
see Lemma~\ref{le:cc} below.
Since $|\g f_j|^{2n}$ is equiintegrable so is $\det B_j \det C_j$ and thus we get
\begin{equation}
 \label{eq:weak_continuity_detBC_L1}
 \det B_j \det C_j \rightharpoonup  0  \quad \text{in $L^1(\Omega)$.}
 \end{equation}

Define
 $$\chi_j(x)  := 
 \begin{cases} 1 & \text{if $\dist(\g f_j(x), L_2) < \dist(\g f_j(x), L_1)$,} \\
 0 & \text{if  $\dist(\g f_j(x), L_2) \ge \dist(\g f_j(x), L_1)$.}
 \end{cases}
 $$ 

 Note that for $\chi_j(x) = 0$ we have $\dist(\g f_j, L_1) = \dist(\g f_j, L)$.
    and thus 
 \begin{equation}  \label{eq:strong_conv_Df_j}
  (1- \chi_j) \dist(\g f_j, L_1) =  (1- \chi_j) \dist(\g f_j, L)
   \to 0 \quad \text{in $L^{2n}(\Omega)$}
 \end{equation}
by   \eqref{eq:convergence_dist_L2n}.
 Thus it suffices to show that 
 \begin{equation} \label{eq:approximate_convergence_chij}
 \chi_j \to 0 \quad \text{in $L^1(\Omega)$.}
 \end{equation}
 Indeed, since $\chi_j$ is a characteristic function, the equiintegrability of $\dist^{2n}(\g f_j, L_1)$ implies that
 \begin{align*}
&  \lim_{j \to \infty}\int_{\Omega}   (1- \chi_j) \dist^{2n}(\g f_j, L_1) \, dx = 0.
 \end{align*}
To show that $\chi_j \to 0$ in $L^1(\Omega)$ 
  we note that  \eqref{eq:strong_conv_Df_j}
 implies
 \begin{equation}  \label{eq:strong_convergence_detBC}
 (1- \chi_j) \det B_j \det C_j \to 0 \quad \text{in $L^1(\Omega)$.}
 \end{equation}
 Combining this with  \eqref{eq:weak_continuity_detBC_L1} we get
  \begin{equation}  \label{eq:weak_convergence_detBC_chi}
 \chi_j \det B_j \det C_j \rightharpoonup 0 \quad \text{in $L^1(\Omega)$.}
 \end{equation}
 By Lemma~\ref{le:minor_facts}~\eqref{it:lemma_minor_off_diagonal}
 \begin{equation}  \label{eq:detF_detBC}
 \chi_j (\det \g f_j - (-1)^n \det B_j \det C_j) \to 0 \quad \text{in $L^1(\Omega)$.}
 \end{equation} and thus
 \begin{equation} \label{eq:weak_convergence_chi_det}
 \chi_j \det \g f_j \rightharpoonup 0 \quad \text{in $L^1(\Omega)$.}
 \end{equation}
 
 We now show that this implies that $\chi_j \to 0$ in $L^1(\Omega)$.
 Recall that $E_{j, \delta} = \{ \det \g f_j < \delta\}$. 
Thus, for any $\delta' > 0$, 
\begin{align*}
& \, \delta' \limsup_{j \to \infty} \int_{\Omega \setminus E_{j, \delta'}} \chi_j  \, dx \\
\le & \,   \limsup_{j \to \infty} \int_{\Omega \setminus E_{j, \delta'}} \chi_j \det \g f_j  \, dx\\
\le & \lim_{ \delta\downarrow  0} \,   \limsup_{j \to \infty} \int_{\Omega \setminus E_{j, \delta}} \chi_j \det \g f_j  \, dx\\
\underset{ \eqref{eq:weak_det_bad}}{\le} & \,   \limsup_{j \to \infty}  \int_{\Omega} \chi_j \det \g f_j  \, dx   \underset{  \eqref{eq:weak_convergence_chi_det}}{=}  \, 0.
\end{align*}
Dividing by $\delta'$ and using the assumption $\lim_{\delta' \downarrow 0} \limsup_{j \to \infty} |E_{j,\delta'}| = 0$ we see that
$\chi_j \to 0$ in $L^1(\Omega)$.

\medskip

 {\bf Step 3:}  \emph{Removal of the additional assumption  \eqref{eq:approx_equiintegrable}.}\\ 
 Now we only assume the hypotheses of Theorem~\ref{th:approximate_solutions}, i.e.,   
 \begin{itemize}
     \item $f_j \rightharpoonup f$ in $W^{1,2n}(\Omega)$,
     \item $\dist (\g f_j, L) \to 0$ in $L^1(\Omega)$,
     \item $\lim_{\delta \downarrow 0} \limsup_{j \to \infty}  | \{ x \in \Omega : \det \g f_j(x) < \delta \}| = 0$.
 \end{itemize} 
We first note that 
it suffices to show that $f$ is split, that $\det \g f > 0$ a.e.\ and there exists an $i \in \{1, 2\}$  such that for a {\it subsequence}
$$
\lim_{k \to \infty} \| \dist (\g f_{j_k}, L_i)\|_{L^q(\Omega)} = 0.
$$
Indeed,  this convergence implies that $\g f \in L_i$ a.e.
Now, if the full sequence $\dist(\g f_j,L_i)$ does not converge to zero in $L^q(\Omega)$ then there exists  another  subsequence 
$m_k$ and an $\eta > 0$ such that $\| \dist(\g f_{m_k}, L_i) \|_{L^q(\Omega)} \ge \eta$. Since the subsequence $k \mapsto f_{m_k}$
still satisfies the assumptions of Theorem~\ref{th:approximate_solutions} there exist a $j \in \{1,2\}$ and a further subsequence
such that $\dist(\g f_{m_{k_l}}, L_j) \to 0$ in $L^q(\Omega)$. Thus $j \ne i$ and $\g f \in L_j$ a.e.
It follows that $\g f \in L_i \cap L_j = \{0\}$ a.e. This contradicts the fact that $\det \g f > 0$ a.e.

We use the following result which assures that (after passage to a subsequence) we may replace $f_j$ by a
sequence $g_j$ which satisfies  \eqref{eq:approx_equiintegrable} and differs from $f_j$ only on  a set whose
measure goes to zero as $j \to \infty$.

\begin{lemma}[\cite{fonseca_muller_pedregal98}, Lemma 1.2] 
\label{le:FMP}  Let $\Omega \subset \R^m$ be bounded and open and let  $f_j$ be a sequence which is bounded in $W^{1,p}(\Omega; \R^d)$. There exists a subsequence $f_{j_k}$  (not relabelled) and sequence $g_k$
such that $|\g g_k|^p$ is equiintegrable and
\begin{equation} 
\lim_{k \to \infty} | \{ g_k \ne f_{j_k} \quad \text{or} \quad \g g_k \ne \g f_{j_k} \}|  = 0.
\end{equation}
\end{lemma}

We apply this lemma with $p = 2n$.

Let $E_k := \{ g_k \ne f_{j_k}  \quad \text{or}  \quad \g g_k \ne Df_{j_k} \}$. Since $|E_k| \to 0$ as $k \to \infty$ and since $\g f_{j_k}$
and $\g g_k$ are bounded in $L^{2n}$   we easily see that 
$$ g_k \rightharpoonup f \quad \text{in $W^{1,2n}(\Omega; \R^{2n})$} \quad \text{and} 
\quad \dist(\g g_k, L) \to 0 \quad \text{in $L^1(\Omega)$}.$$
Moreover, since $|E_k| \to 0$ we get from  assumption   \eqref{eq:lower_bound_det}
$$ \lim_{\delta \downarrow 0} \limsup_{j \to \infty}  |\{ x \in \Omega : \det \g g_k < \delta \} | = 0.$$

Thus our previous reasoning in Step 1 and Step 2  applies to the sequence  $g_k$  and we deduce that the  weak limit $f$ satisfies $\det \g f > 0$,
$f$ is globally split, and 
$$ \dist(\g g_k, L_i) \to 0 \quad \text{in $L^{2n}(\Omega)$} $$
for $i=1$ or $i=2$.
Since $\dist(\g f_{j_k}, L_i)$ is bounded in $L^{2n}$ and $|E_k| \to 0$ we obtain
$\dist(\g f_{j_k}, L_i) \to 0$ in $L^q(\Omega)$ for all $q < 2n$. This concludes the proof of 
Theorem~\ref{th:approximate_solutions}.
\end{proof}

\begin{lemma} \label{le:cc}
Let $\Omega \subset \R^{2n}$ be bounded and open. Assume that $f_j \rightharpoonup f$ in $W^{1,2n}(\Omega, \R^{2n})$
and $\dist(\g f_j, L) \to 0$ in $L^{2n}$. Express $\g f_j$ as a block-diagonal matrix,
$$ Df_j = \begin{pmatrix} A_j & B_j \\ C_j & D_j  \end{pmatrix}$$
and similarly for $\g f$. Then
\begin{eqnarray} \label{eq:weak_continuity_detA_detD}
\det A_j \det D_j &\overset{*}{\rightharpoonup} \det A \det D     \quad \text{in $\mathcal M(\Omega)$,} \\
\det B_j \det C_j &\overset{*}{\rightharpoonup} \det B \det C  \label{eq:weak_continuity_detB_detC}
 \quad \text{in $\mathcal M(\Omega)$.}
\end{eqnarray}
\end{lemma}

\begin{proof} Let $M$ be an $n \times n$ minor which vanishes on $L$. Then 
Lemma~\ref{le:minor_facts}~\eqref{it:lemma_minor_converge} implies that 
$M(Df_j) \to 0$ in $L^2(\Omega)$. In particular the pullbacks of the $n$-forms 
$\alpha_1 = dy_1 \wedge \ldots \wedge dy_n$ and 
$\alpha_2 = dy_{n+1} \wedge \ldots \wedge dy_{2n}$ satisfy
\begin{eqnarray*}
 R^{(1)}_j := f_j^* \alpha_1 - \det A_j \,  \alpha_1 - \det B_j \,  \alpha_2 &\to& 0 \quad \text{in $L^2(\Omega)$,} \\
 R^{(2)}_j := f_j^* \alpha_2 - \det C_j \,  \alpha_1 - \det D_j \,  \alpha_2 &\to& 0 \quad \text{in $L^2(\Omega)$.}
 \end{eqnarray*}
 We argue as in the proof of Theorem~\ref{th:split}. Thus, for any $k\geq n+1$ we define the $(n-1)$-forms $\omega^{(k)}=dx_1\wedge\dots\wedge\widehat{dx_k}\wedge\dots\wedge dx_{2n}$ (i.e.~$dx_{k}$ missing), and $\beta^{(k)}=\varphi\omega^{(k)}$, with $\varphi\in C_c^\infty(\Omega)$. Then 
 $$
 d\beta^{(k)}=\sum_i\frac{\partial\varphi}{\partial x_i}dx_i\wedge\omega^{(k)},
 $$
 where the sum is over $i\in\{1,\dots,n,k\}$. We apply Lemma~\ref{le:pullback_closed} to deduce 
 $$
 \int f_j^*\alpha_1\wedge d\beta^{(k)}=0,
 $$
 hence
 $$
\int R_j^{(1)}\wedge d\beta^{(k)}=(-1)^k\int \det A_j\frac{\partial\varphi}{\partial x_k}\,dx\to 0\textrm{ as }j\to\infty
 $$
 for any $\varphi\in C_c^\infty(\Omega)$ and consequently, by density and the uniform $L^2$ bound on the sequence $\det A_j$, for any 
 $\varphi \in W^{1,2}_0(\Omega)$. 
 In other words
 \begin{equation}  \label{eq:H-1_A}
 \frac{\partial \det A_j}{\partial x_k} \to 0 \quad \text{in $H^{-1}(\Omega)$ \quad  for $k \ge n+1$.}
 \end{equation} 
 Here $H^{-1}(\Omega)$ denotes the  dual space of $H^1_0(\Omega) := W^{1,2}_0(\Omega)$, the closure of $C_c^\infty(\Omega)$
 in $W^{1,2}(\Omega)$. 
 Similarly we get
\begin{eqnarray}
 \frac{\partial \det C_j}{\partial x_k}  &\to& 0 \quad \text{in $H^{-1}(\Omega)$ \quad  for $k \ge n+1$,} \\
  \frac{\partial \det B_j}{\partial x_k}  &\to& 0 \quad \text{in $H^{-1}(\Omega)$ \quad  for $k \le n$,}\\
   \frac{\partial \det D_j}{\partial x_k}  &\to& 0 \quad \text{in $H^{-1}(\Omega)$ \quad  for $k \le n$.}   \label{eq:H-1_D}
 \end{eqnarray}
 Moreover by Lemma~\ref{le:weak_continuity} 
 \begin{eqnarray} 
  \det A_j  &\rightharpoonup& \det A, \quad  \det B_j  \rightharpoonup \det B \quad \text{in $L^2(\Omega)$,}   \label{eq:weak_detA} \\
 \det C_j  &\rightharpoonup& \det C, \quad  \det D_j  \rightharpoonup \det D \quad \text{in $L^2(\Omega)$.} \label{eq:weak_detD} 
 \end{eqnarray}
 
 The assertion now follows from  \eqref{eq:H-1_A}--\eqref{eq:weak_detD}  and the theory of compensated compactness, developed by
  Murat and Tartar,
  see for instance
  \cite{murat77, murat78,  tartar79, tartar15}.
  
 To see this, consider a map $g: \Omega \to \R^2$ and the first-order constant coefficient differential operator
 $\mathcal A: L^2(\Omega; \R^2) \to H^{-1}(\Omega; \R^{2n})$ given by
 \begin{eqnarray*}(\mathcal A g)_k &=& \frac{\partial g_2}{\partial x_k} \quad \text{for $1 \le k \le n$}, \\
  (\mathcal A g)_k &=& \frac{\partial g_1}{\partial x_k} \quad \text{for $n+1 \le k \le 2n$.}
  \end{eqnarray*}
  Associated to $\mathcal A$ is the \emph{wave cone} $\Lambda$ of ``dangerous amplitudes'' defined by
  $$ \Lambda := \{ a \in \R^2 :  \exists \xi \in \R^{2n} \setminus \{ 0\} \, \,  \quad  \mathcal A (a e^{i x \cdot \xi}) = 0 \}.$$
  One easily checks that for the operator $\mathcal A$ defined above
  $$ \Lambda:= \{ (a_1, 0) : a_1 \in \R \} \cup \{ (0, a_2) : a_2 \in \R \}. $$
  Let $Q$ be a quadratic form that vanishes on $\Lambda$. The theory of compensated  compactness 
  implies that  \cite[Thm. 11]{tartar79}
  \begin{align*}
 & \quad  g_j \rightharpoonup   g \quad \text{in $L^2(\Omega)$}  \quad \text{and} \quad  \mathcal A g_j \to 0  \quad  \text{in $H^{-1}(\Omega)$} \\
  \Longrightarrow & \quad 
   Q(g_j) \overset{*}{\rightharpoonup} Q(g) \quad \text{in $\mathcal M(\Omega)$.}
\end{align*} 
Applying this with $Q(a) = a_1 a_2$ and $g_j = (\det A_j, \det D_j)$ or $g_j = (\det B_j, \det C_j)$ we obtain
\eqref{eq:weak_continuity_detA_detD} and \eqref{eq:weak_continuity_detB_detC}
\end{proof}

\bigskip

We now discuss two examples which show that
 the condition  
$$
\lim_{\delta \downarrow 0} \limsup_{j \to \infty}  | \{ x \in \Omega : \det \g f_j(x) < \delta \}| = 0
$$
 in   Theorem~\ref{th:approximate_solutions}  
 cannot be replaced by the condition 
$\det \g f_j > 0$ or the condition $|\det \g f_j| \ge \delta > 0$.

\begin{example}  \label{ex:counterexample_det_bigger_zero} 
To see that the condition $\det \g f_j >  0$ is not sufficient, 
consider the case $n=2$. 
By Theorem~\ref{th:nosplit} there exists a map $F: (0,1)^2 \subset \R^2 \to \R^2$
such that  
$$
\g F = \begin{pmatrix} a & b \\ c & d\end{pmatrix}
$$ 
is split and satisfies $\det \g F = 1$ a.e. 
For any $\eps\geq 0$ define $f^{(\eps)}:(0,1)^4\to\R^4$ by
$$ f^{(\eps)}_1(x) = F_1(x_1, x_3), \quad f^{(\eps)}_3(x) = F_2(x_1, x_3), $$
$$ f^{(\eps)}_2(x) = \eps x_2  - \eps x_4, \quad
f^{(\eps)}_4(x) = \eps x_2 + \eps x_4.$$
Then 
$$ 
\g f^{(\eps)}(x) = \begin{pmatrix}  a & 0 & b & 0 \\
0 & \eps & 0 & -\eps\\
c & 0 & d & 0 \\
0 &\eps & 0 &  \eps 
\end{pmatrix}.
$$
By swapping rows and columns 2 and 3, we see that for a.e.~$x$
$$
\det \g f^{(\eps)}=\det \begin{pmatrix}  \g F&0\\ 0&\eps X \end{pmatrix}=\det \g F\det(\eps X)=2\eps ^2, 
$$
where $X=\begin{pmatrix} 1&-1\\ 1&1\end{pmatrix}$ and we have written the $4\times 4$ matrix in block matrix form. Furthermore, since $\g F$ is split a.e.~as a linear map on $\R\times\R$, also $\g f^{(0)}$ is split a.e.~as a linear map on $\R^2\times\R^2$, and hence 
$$ 
\dist(\g f^{(\eps)}(x), L) \leq \dist(\g f^{(0)}(x), L) + c\eps
$$
Then we have
$$
f^{(\eps)}\to f^{(0)}\quad \textrm{ in }W^{1, \infty},
$$
but the limit map $f^{(0)}$ is not globally split since $F$ is 
not globally split. 
\end{example}

\begin{example}  \label{ex:det_pm_approximate}
Let $\Sigma_\pm = \{ X \in \R^{2n \times 2n} : |\det X| = 1\}$. Let $\Omega = (0,1)^{2n}$. 
We show that there exists a finite set $E \subset L \cap \Sigma_\pm$ and a sequence of uniformly Lipschitz maps such that
\begin{equation}  \label{eq:det_pm_approximate1}  f_j \overset{*}{\rightharpoonup} f \quad \text{in $W^{1, \infty}(\Omega; \R^{2n})$}, 
\quad 
\dist(\g f_j, E) \to 0 \in L^\infty(\Omega)
\end{equation}
and 
\begin{equation} \label{eq:det_pm_approximate2} \g f \equiv \frac14 e_1 \otimes e_1 + \frac14 e_{n+1} \otimes e_1 \notin L.
\end{equation}

The construction is based on so called laminates of finite order which are defined as follows.
Let $\nu$ be a probability measure on $\R^{d \times m}$ which is supported on a finite set, $\nu = \sum_{i=1}^r \lambda_i \delta_{A_i}$
with $A_i \ne A_j$ if $i \ne j$. 
We say that $\nu'$ is obtained from $\nu$ by splitting if there exist  $j \in \{1, \ldots, r\}$, $s \in [0, \lambda_j]$ and 
matrices $B', B''$ such that 
\begin{align*}
\nu' & \, = \nu +  s  \, \big(  \lambda \delta_{B'} + (1- \lambda) \delta_{B''} - \delta_{A_j}  \big),
\end{align*}
where
\begin{align*}\quad A_j & \, = \lambda B' + (1-\lambda) B'', \quad \rank (B''-B') = 1.
\end{align*}
Note that $\nu$ and $\nu'$ have the same center of mass $\overline \nu = \overline{\nu'}  =\sum_{i=1}^r  \lambda_i A_i$.
We say that a probability measure on $\R^{d \times m}$ is a  laminate of finite order  if it can be obtained from a Dirac mass by a finite number of splittings.

We will show that there exists a laminate $\nu$ of finite order with $\bar \nu = 
 \bar A := \frac14 e_1 \otimes e_1 + \frac14 e_{n+1} \otimes e_1$
 which is given by 
$\nu = \sum_{i=1}^r \lambda_i \delta_{A_i}$ with $A_i \ne A_j$ for $i \ne j$ and $A_i \in L \cap \Sigma_\pm$. 
Set $E = \{ A_1, \ldots A_r\}$.

Then by \cite[Lemma 3.2]{muller_sverak03}  there exist  piecewise affine Lipschitz  maps
$f_j : \Omega \to \R^{2n}$ such that 
\begin{itemize}
    \item $f_j(x)=\bar{A}x$ on $\partial\Omega$,
    \item $|f_j(x)-\bar{A}x|<2^{-j}$ in $\Omega$,
    \item $\dist(\g f_j, E) < 2^{-j}$ a.e.~in $\Omega$.
\end{itemize}
Then, since the sequence is uniformly Lipschitz and bounded in $W^{1,\infty}(\Omega)$, by the Banach-Alaoglu theorem we may assume in addition and without loss of generality (by passing to a subsequence if necessary) that $f_j \overset{*}{\rightharpoonup} f$ in $W^{1, \infty}(\Omega)$, and moreover $f(x)\equiv \bar{A}x$, as claimed in  \eqref{eq:det_pm_approximate1}- \eqref{eq:det_pm_approximate1}

To construct $\nu$, we first construct a laminate $\nu_1$ of finite order with $\bar \nu_1 = \frac12 e_1 \otimes e_1$
which is supported on  the set $E_1 = \{  \sum_{i=1}^{2n} \sigma_i e_i \otimes e_i : \sigma_i \in \{-1,1\} \}$
of diagonal matrices with entries $\pm 1$. 
To do so,  we write $(a_1, \ldots, a_{2n}) := \sum_{i=1}^{2n} a_i e_i \otimes e_i$ and use the splittings
\begin{eqnarray*}
\delta_{(\tfrac12, 0, \ldots, 0)} & \lra & \frac14  \delta_{(-1, 0, \ldots, 0)} + \frac34 \delta_{(1,0, \ldots, 0)}, \\
\delta_{(\pm 1, 0, \ldots, 0)} & \lra & \frac12  \delta_{(\pm 1, -1, 0, \ldots, 0)} + \frac12 \delta_{(\pm 1, 1, 0 , \ldots, 0)}, \\
&\ldots,  & \\
\delta_{(\pm1 ,  \ldots, \pm 1,  0)} & \lra & \frac12  \delta_{(\pm 1, \ldots, \pm 1,  -1)} + \frac12 \delta_{(\pm1, \ldots, \pm 1, 1)}. \\
\end{eqnarray*}
Similarly we obtain a laminate $\nu_2$ of finite order with $\bar \nu_2 = \frac12 e_{n+1} \otimes e_1$. Specifically, we can consider the linear  map 
$P : \R^{2n} \to \R^{2n}$ given by $P\binom{x'}{x''} = \binom{x''}{x'}$ for $x', x'' \in \R^n$. Let $\ell_P$ denote  the action of $P$  on $\R^{2n \times 2n}$ by left multiplication
of matrices. Then $\ell_P$ preserves the set $L \cap \Sigma_\pm$.   Moreover, pushforward by $\ell_P$ maps laminates of finite order to laminates of finite 
order,  since $\ell_P$ preserves rank-one lines. Hence the pushward measure $\nu_2 = \ell_P^* \nu_1$
 is a laminate of finite order which is supported on $L \cap \Sigma_\pm$
and satisfies $\bar \nu_2 =\frac12 e_{n+1} \otimes e_1$. 
Finally,  using the splitting
$$ \delta_{\tfrac14 e_1 \otimes e_1 + \tfrac14 e_{n+1} \otimes e_1} \, \lra \, \frac12 \delta_{\tfrac12 e_1 \otimes e_1} + \frac12 \delta_{\tfrac12 e_{n+1} \otimes e_1}$$
shows that $\nu = \frac12 \nu_1 + \frac12 \nu_2$ is a laminate of finite order  which is supported on $L \cap \Sigma_\pm$ and satisifies
$\bar \nu = \frac14 e_1 \otimes e_1 + \frac14 e_{n+1} \otimes e_1$.
\end{example}

\section{No global splitting for $n = 1$ -- overview of the argument}  \label{se:nosplit}
In this section we give a short  overview  of the  argument to prove  Theorem~\ref{th:nosplit}.
We first note that the Theorem is an immediate consequence of Proposition~\ref{prop:nosplit} below.
In the following two sections we develop the needed auxiliary results in detail to prove Proposition~\ref{prop:nosplit}.
In fact, we give a more precise statement of the result as Proposition~\ref{prop:largeT5ex}. This proposition is then proved in 
Section~\ref{se:T4}.


Recall that the set of split matrices is given by $L=L_1 \cup L_2$ with
\begin{equation}
L_1 := \left\{  \begin{pmatrix} a & 0 \\ 0 & d 
\end{pmatrix} : a, d \in \R \right\}, \quad 
L_2 := \left\{  \begin{pmatrix} 0 & b \\ c & 0 
\end{pmatrix} : b,c \in \R \right\}
\end{equation}
We set
\begin{equation}
\Sigma := \{ Y\in\R^{2\times 2} : \det Y = 1\}.
\end{equation}
Let $\Omega_1, \Omega_2$ be bounded and open intervals in $\R$ and set $\Omega=\Omega_1\times\Omega_2$.
Our aim is to prove the following statement:  

\begin{proposition}\label{prop:nosplit}
There exists a compact set $K \subset L\cap\Sigma$, a 
matrix $A \in \Sigma \setminus L$ and a Lipschitz map $f:\overline{\Omega}\to\R^2$ such that
\begin{eqnarray}   \label{eq:Df_in_K0}  
\g f(x) &\in & K \quad \text{for a.e. $x \in \Omega$ and}\\
 \label{eq:affine_boundary0}
f(x) &=& Ax \quad \text{for  all $x \in \partial \Omega$.}
\end{eqnarray}	
Moreover, $K$ can be chosen to consist of 5 elements.
\end{proposition}

Indeed, Theorem \ref{th:nosplit} follows immediately from this proposition.

\begin{proof}[Proof of Theorem~\ref{th:nosplit}]
Let $f$ be as in Proposition~\ref{prop:nosplit}.
Then $\g f(x)$ is split with $\det \g f(x) = 1$ for almost every $x\in\Omega$, but $f$ is not globally split because $A$ is not split. 

Note that  the restriction of $f$ to $\partial \Omega$ is affine.  To see that $f$ is bi-Lipschitz on $\overline \Omega$ one
can either use Theorem 2 in \cite{ball_invertibility81} or use the theory of quasiregular mappings \cite{ahlfors66} as follows. First of all, setting $\mathcal{K}= \| \g f\|_{L^\infty}^2$ we see that $|\g f|^2 \le K \det \g f$
a.e. and thus $f$ is $\mathcal{K}$-quasiregular (or, equivalently, a map of bounded distorsion). Being affine on the boundary, it then follows from \cite[Theorem 5]{LeonettiNesi} or \cite[Theorem 6.1]{BojarskiDOnofrioetal}
that $f$ is a homeomorphism, hence $\mathcal{K}$-quasiconformal. 

Then $f^{-1}$ is also $\mathcal{K}$-quasiconformal
(see, e.g., \cite[Chapter 2]{ahlfors66}) 
and in particular in the Sobolev space $W^{1,1}(A \Omega)$ (which is equivalent to the space $\mathrm{ACL}(A\Omega))$.
In particular $\g f^{-1} = (\g f)^{-1} \circ f^{-1}$ a.e. and thus $\g f^{-1}$ is in $L^\infty(A\Omega)$. Since $A \Omega$ is Lipschitz domain
(in fact parallelogram),  $f^{-1}$ is Lipschitz on $A\Omega$ and hence has a Lipschitz extension to the closure.
%
%
\end{proof}

\bigskip

The proof of Proposition \ref{prop:nosplit} will be given in the next two sections. In a nutshell, to construct a map which satisfies  \eqref{eq:Df_in_K0}  and  \eqref{eq:affine_boundary0}, we use
the theory of convex integration for  Lipschitz maps. After briefly reviewing the theory in Section \ref{se:convex_integration}, we restate the proposition more precisely as Proposition \ref{prop:largeT5ex}. The proof will then be given in Section \ref{se:T4}.

The key challenge in our setting is the lack of rank-one connections in the set $L\cap \Sigma$; that is, for any $A,B\in L\cap \Sigma$ with $A\neq B$ we have $\textrm{rank}(A-B)=2$. The significance of this property lies in the following standard construction (c.f. with the `folding map' example described in the introduction):

\begin{example}\label{ex:rankoneAB}
Let $A,B\in\R^{d\times m}$ with $\textrm{rank}(A-B)=1$. Such pairs of matrices are referred to as \emph{rank-one connections}. Then we can write $A-B=a\otimes\xi$ for some $a\in \R^d$, $\xi\in \R^m$; further, let $C=\frac{1}{2}(A+B)$. Given any Lipschitz function $h:\R\to\R$ with $h'(t)\in\{+1,-1\}$ a.e., set $f(x):=Cx+\frac12ah(x\cdot\xi)$. Then $f:\R^m\to\R^d$ is Lipschitz with $\g f(x)\in\{A,B\}$ a.e.
\end{example}
In other words the presence of such rank-one connections $A,B\in K$ allows Lipschitz solutions of the corresponding differential inclusion \eqref{eq:Df_in_K0} to `combine' the two gradients $A,B$. Despite the very simple nature of this construction, the question whether or not rank-one connections exist in any given set $K\subset\R^{d\times m}$ has played a pivotal role in the theory of differential inclusions of the type \eqref{eq:Df_in_K0}, with far-reaching consequences. For instance, if $K$ is a $C^1$ submanifold in $\R^{d\times m}$, non-existence of rank-one connections in the tangent spaces $T_AK$ can be identified with a form of (linearized) ellipticity in the sense of Legendre-Hadamard (see \cite{scheffer74,Ball:1980fy,Sverak:1993va,muller_sverak03,muller99,szekelyhidi04,Lorent:2019kv}). In particular in our $2\times 2$ setting, the lack of rank-one connections in both $K$ and its tangent spaces leads to higher regularity, as shown by \v Sver\'ak \cite{Sverak:1993va}, provided $K$ is connected: in that case every Lip\-schitz solution $f$ of \eqref{eq:Df_in_K0} in fact belongs to $C^{1,\alpha}$ and moreover if $K$ is smooth, so is $f$. At this point it is worth noting that our set $E:=L\cap\Sigma\subset\R^{2\times 2}$ satisfies the following properties, both of which are easy to verify:
\begin{itemize}
\item The set $E$ is the disjoint union of two smooth `elliptic' sets $E=E_1\cup E_2$ with $E_i=L_i\cap\Sigma$. That is, for each $i=1,2$ the set $E_i$ is a smooth curve with tangent directions given by rank-two matrices;
\item The set $E$ contains no rank-one connections. That is, for any $A,B\in E$ with $A\neq B$, $\textrm{rank}(A-B)=2$.
\end{itemize}
Although ellipticity leads to higher regularity for $C^1$ solutions, it does not exclude the possibility of large jumps in the gradient $\g f$ for Lipschitz solutions, \emph{even in the absence of rank-one connections}. To explain this in some detail, let us first consider the construction of approximately split maps (c.f.~Theorem \ref{th:approximate_solutions}):
\begin{proposition}\label{prop:noapproxsplit}
There exists a compact set $K \subset L\cap\Sigma$, a 
matrix $A \in \Sigma \setminus L$ and a sequence of uniformly Lipschitz maps $f_j:\overline{\Omega}\to\R^2$ such that
\begin{eqnarray}   \label{eq:Df_to_K0}  
\dist(\g f_j,K) &\to & 0 \quad \text{in $L^\infty(\Omega)$}\\
 \label{eq:affine_boundary01}
f_j(x) &=& Ax \quad \text{for  all $x \in \partial \Omega$.}
\end{eqnarray}	
\end{proposition}
The proof of Proposition \ref{prop:noapproxsplit} is based on the observation that one can find special 4-element sets $K=\{X_1,\dots,X_4\}\subset L\cap\Sigma$ forming a so-called $T_4$-configuration - see Definition \ref{de:T4} below in Section \ref{se:T4}. Such configurations, discovered independently by a number of authors in various contexts \cite{scheffer74,aumann_hart86,nesi_milton91,casadio_tarabusi93,tartar93},  have played a central role in understanding the proper generalisation of Example \ref{ex:rankoneAB}, and in particular in the work of Scheffer \cite{scheffer74} and subsequently also in \cite{muller_sverak03,szekelyhidi04} to produce counterexamples to regularity for elliptic systems. In our situation the precise result, whose proof will be given in Section \ref{sec:T4ex}, is the following:
\begin{lemma}\label{lem:T4}
Let $c>1$ and 
\begin{equation}
\begin{split}
	X_1&=\begin{pmatrix}c&0\\0&1/c\end{pmatrix}, \quad  X_2=\begin{pmatrix}1/c&0\\0&c\end{pmatrix},\\ X_3&=\begin{pmatrix}0&-c\\1/c&0\end{pmatrix}, 
	\quad X_4=\begin{pmatrix}0&-1/c\\c&0\end{pmatrix},
	\end{split}
\end{equation}
so that $X_1,X_2\in L_1\cap \Sigma$ and $X_3,   X_4\in L_2\cap\Sigma$. Then $(X_1,X_2,X_3,X_4)$ is a $T_4$ configuration if and only if
 $c>1+\sqrt{2}$.
 \end{lemma}

Consequently, such a $T_4$-configuration $K=\{X_1,\dots,X_4\}$ satisfies
 the conclusions of Proposition \ref{prop:noapproxsplit} (see Section \ref{se:T4}). 
 Although such a $4$-element set cannot work in Proposition  \ref{prop:nosplit} (see \cite{Chlebik:2002va}), 
 it is possible to adapt the stability argument of \cite{muller_sverak03} to show that for sufficiently small $\varepsilon>0$ the set 
$$
K':=\bigl\{X\in L\cap\Sigma:\,\dist(X,K)\leq \varepsilon\bigr\}
$$   
satisfies the conclusions of Proposition  \ref{prop:nosplit}.
An alternative approach, based on \cite{szekelyhidi_forster18}, is to find suitable $T_5$ configurations $(X_1,\dots,X_5)$ in $L\cap\Sigma$. 
Since the latter has, to the best of our knowledge, not been applied 
for concrete differential inclusions so far and hence may be of independent interest, 
we opt in this paper to present the details of this alternative approach in the next sections. The precise result is stated in Proposition \ref{prop:largeT5ex}.

\section{Convex integration}  \label{se:convex_integration}

 In this section we review some results from  the theory of convex integration  which are required for the proof of Proposition~\ref{prop:nosplit}.
Let $E$ be a subset of the $ d\times m$ matrices  $\R^{d \times m}$ and let $A \in \R^{d \times m}$.
Let $\Omega \subset \R^m$ be bounded and open. 
Convex integration provides sufficient conditions for the existence of a Lipschitz map $f: \overline \Omega \to \R^d$
such that
\begin{eqnarray}  \label{eq:Df_in_E} \g f(x) &\in& E \quad \text{for a.e. $x \in \Omega$ and}\\
 \label{eq:affine_boundary}
f(x) &=& Ax \quad \text{for $x \in \partial \Omega$.}
\end{eqnarray}

In fact,  convex integration does much more. It shows that the affine function $x \mapsto Ax$, viewed as  function 
on $ \Omega$, admits a fine $C^0$-approximation by functions $f$ with $\g f \in E$ a.e., i.e. for every continuous function 
$\eps: \Omega \to (0, \infty)$ there exists a map $f$  with $\g f \in E$ a.e.\ such that $|f(x) - Ax| < \eps(x)$. Taking  $\eps(x) = \eps_0 \dist(x, \partial
\Omega)$ we recover  \eqref{eq:affine_boundary}.
More generally,  any $C^1$ function $g: \overline{\Omega} \to \R^d$ with $\g g$ in a suitable set $E'$ admit a fine $C^0$ approximation by functions $f$ with $\g f \in E$ a.e.
For our purposes functions which satisfy  \eqref{eq:Df_in_E}  and \eqref{eq:affine_boundary} are sufficient, so we focus on this setting. 

Roughly speaking, convex integration asserts that the problem \eqref{eq:Df_in_E},  \eqref{eq:affine_boundary}
can be solved if $A$ lies in a suitable convex hull of  $E$. 
The key idea of convex integration is to 'deform' affine functions by adding 
increasingly faster one-dimensional oscillations of the type given in Example \ref{ex:rankoneAB}, which 'move' the gradient closer to the set $E$. Then one uses a careful limiting argument to ensure that in this process the gradients converge 
strongly.

This general strategy originates in the seminal work of Nash on $C^1$ isometric embeddings \cite{nash54}, 
which was subsequently extended and developed by 
Gromov \cite{gromov_pdr} into the far-reaching and powerful technique 
of convex integration. 
Although the technique was originally intended to deal with under-determined problems
 in geometry and topology, more recently the same
  ideas have been extended to various systems of partial differential equations arising 
  in continuum mechanics, most notably nonlinear
   elasticity \cite{Kirchheim:2002wc} and hydrodynamics \cite{SzekelyhidiJr:2014tu,delellis_szekelyhidi16,Buckmaster:2019et}. 

For many of the applications it suffices to consider the lamination convex hull $E^{lc}$ 
(this essentially corresponds to Gromov's $P$-convex hull \cite{gromov_pdr}). 
We recall briefly that a set is called lamination convex if for 
any rank-one connection  $A,B\in E$ (i.e. with $\textrm{rank}(A-B)=1$ the whole line segment $[A,B]$ is contained in $E$; 
and the lamination convex hull is the smallest lamination convex set containing $E$. 
In our setting the
 set $E$ contains no rank-one connections,  and hence is automatically lamination convex.
 The key point is that in this setting  one can to work with the potentially much larger
 rank-one convex hull,  defined by duality with rank-one convex functions.

\begin{definition}  \label{de:rc} A function $g: \R^{d \times m} \to \R$ is rank-one convex if it is convex along any line whose
direction is given by a matrix of rank-one. 
For a compact set $K \subset \R^{d \times m}$ the rank-one convex hull is defined as the set of points
which cannot be separated from $K$ by rank-one convex functions, i.e.,
\begin{align}  K^{rc} := \{& F \in \R^{d \times m} :  
\text{$g(F) \le 0$ whenever} \\   & \text{ $g|_K \le 0$ and $g$ is  rank-one convex.} \} \nonumber 
\end{align}
For an open set $U \subset \R^{d \times m}$ we define
\begin{equation}
U^{rc} = \bigcup_{\text{$K \subset U$ compact}}  K^{rc}.
\end{equation}
\end{definition}

We note that for ordinary convexity the definition of the convex hull via separation by convex
functions is equivalent to the definition by considering convex combinations. 
This is not true for rank-one convexity. In fact our analysis below relies heavily on certain
finite sets ('$T_N$ configurations', see Section~\ref{se:T4}) which have a nontrivial rank-one convex hull, 
but contain no rank-one connections.

The first key result in convex integration theory, relying on an iterated construction based on Example \ref{ex:rankoneAB}, is that for open sets $E$ the problem 
 \eqref{eq:Df_in_E},  \eqref{eq:affine_boundary} can be solved if $A \in E^{rc}$, see e.g., 
 \cite[Thm. 3.1]{muller_sverak03}. For many applications, including the case of split matrices $L$, this is not sufficient because in such cases the set $E$ is a closed, lower-dimensional subset. Furthermore, in our setting of $E=L\cap\Sigma$, not just $E$ but also $E^{rc}$ is lower-dimensional; indeed, observe that the functions $X\mapsto \pm\det X$ are rank-one convex (in fact rank-one affine), and consequently $(L\cap\Sigma)^{rc}\subset\Sigma$.   There are several methods to pass from open sets to closed (lower-dimensional) sets, see, e.g., 
 \cite{muller_sverak96, dacorogna_marcellini99, sychev01, kirchheim03, muller_sverak03} . In particular the constraint $(L\cap\Sigma)^{rc}\subset\Sigma$ has been treated in \cite{muller_sverak_constraint}. The main result of \cite{muller_sverak_constraint}, specialized to our setting \eqref{eq:Df_in_K0}-\eqref{eq:affine_boundary0}, reads as follows.

\begin{definition}[\cite{muller_sverak_constraint}, Def. 1.2.] \label{de:in_approx_constraint}
 Let 
\begin{equation}\label{eq:set_constraint}
 \Sigma := \{ X \in \R^{2 \times 2} : \det X = 1 \}.
 \end{equation}
and let $K \subset \Sigma$ be compact.
We say that a sequence of sets $U_i \subset \Sigma$ is an in-approximation relative to $\Sigma$  if the sets
$U_i$  are open in $\Sigma$, and the following two conditions are satisfied
\ben
\item $U_i \subset U_{i+1}^{rc}$; 
\item  \label{it:inapprox_det}  $\lim_{i \to \infty} \sup_{X \in U_i} \dist(X, K) = 0$.
\een
\end{definition}

Recall that a set $U \subset \Sigma$ is open in $\Sigma$ if there exists and open set $V \subset \R^{2 \times 2}$ such that
$U = \Sigma \cap V$. In \cite[Def. 1.2]{muller_sverak_constraint} the additional assumption that the $U_i$ be uniformly bounded is made. 
If $K$ is bounded this follows from property   \eqref{it:inapprox_det} in Definition~\ref{de:in_approx_constraint}.

\begin{theorem}[\cite{muller_sverak_constraint}, Thm.\ 1.3] \label{th:constraint}
Let $\Omega \subset \R^2$ be open, bounded and connected. Let $\Sigma$ be given by \eqref{eq:set_constraint}
 and let  $K \subset \Sigma$ be compact. Let $U_i$ be an in-approximation of $K$ relative to $\Sigma$ and 
assume that 
$$ A \in U_1.$$
Then there exists a Lipschitz map $f: \Omega \to \R^2$ such that 
$$ \g f(x) \in K \quad \text{for a.e.\ $x \in \Omega$} $$
and 
$$ f(x) = Ax \quad \text{for all $x \in \partial \Omega$.} $$
\end{theorem}

Our main result concerning split matrices, which will be proved in the next section, is the following
\begin{proposition}\label{prop:largeT5ex}
Let $c\geq 3$ and define the matrices 
\begin{equation}\label{T5e:ex}
\begin{split}
X_1&=\begin{pmatrix} c&0\\0&1/c\end{pmatrix},\, 
X_2=\begin{pmatrix} 1/c&0\\0&c\end{pmatrix},\\
X_3&=\begin{pmatrix} 0&-c\\1/c&0\end{pmatrix},\,
X_4=\begin{pmatrix} 0 &-1/c\\ c & 0\end{pmatrix},\,
X_5=\begin{pmatrix} 1&0\\0&1\end{pmatrix},	
\end{split}
\end{equation}
so that $X_1,X_2,X_5\in L_1\cap\Sigma$ and $X_3,X_4\in L_2\cap\Sigma$. Then the set $K=\{X_1,\dots,X_5\}$ admits an in-approximation relative to $\Sigma$. Consequently there exists $A\in \Sigma\setminus L$ such that for any bounded open $\Omega\subset\R^2$ there exist Lipschitz maps $f:\overline{\Omega}\to\R^2$ with
\begin{align*}
\g f(x)&\in K\quad \textrm{for a.e. }x\in\Omega,\\
f(x)&=Ax\quad\textrm{ for all }x\in\partial\Omega.
\end{align*}  
\end{proposition}

The proof of proposition, which will be given in   Section~\ref{sec:largeT5ex} follows 
 the strategy introduced in \cite{szekelyhidi_forster18} for the construction of an in-approximation, 
 which is based on proving that the set $\{X_1,\dots,X_5\}$ is a \emph{large $T_5$ set}, see Definition \ref{T5d:largeT5}.

\bigskip
\section{$T_N$ configurations}  \label{se:T4}

\subsection{Definition and a criterion for $T_N$ configurations in $2 \times 2$ matrices}

\begin{definition}  \label{de:T4}
Let $ N \ge 4$. An $N$-tuple $(X_1, \ldots,  X_N)$ of matrices in $\R^{m \times n}$ is called a $T_N$ configuration if 
$\rank(X_i - X_j ) > 1$ for $i \ne j$ and if there exist matrices $P,C_1, \ldots, C_N \in  \R^{m \times n}$
 and real numbers $\kappa_1,  \ldots, \kappa_N > 1$  such that
and
\begin{eqnarray*}
X_1 &=& P + \kappa_1 C_1, \\
X_2& =& P + C_1 + \kappa_2 C_2,\\
\vdots \\
X_N &=& P + C_1 +  \ldots  + C_{N-1} +  \kappa_N C_N
\end{eqnarray*}
and 
$$ \rank C_i = 1, \quad \sum_{i=1}^N C_i = 0.$$
\end{definition}

\begin{figure}
\begin{tikzpicture}[scale=1]
\def\PA{(2,0)}
\def\PB{(4,1)}
\def\PC{(3, 2.5)}
\def\PD{(1.5,2.9)}
\def\PE{(0.5, 1.6)}
\def\TA{\PA +  (1.5*2,1.5)}
\def\TB{\PB +  (-1.5,  1.5*1.5)}
\def\TC{\PC +  (-1.5*1.5,1.5*0.4)}
\def\TD{\PD +  (-1.5,-1.5*1.3)}
\def\TE{\PE +  (1.5*1.5,-1.5*1.6)}
\draw \TA -- \PA;
\draw \PA -- \TA;
\draw \TB -- \PB;
\draw \TC -- \PC;
\draw \TD -- \PD;
\draw \TE -- \PE;
\filldraw \PA circle (1pt)  node[anchor=north] {$P_1$};
\filldraw \PB circle (1pt)  node[anchor=north] {$P_2$};
\filldraw \PC circle (1pt)  node[anchor=west] {$P_3$};
\filldraw \PD circle (1pt)  node[anchor=south] {$P_4$};
\filldraw \PE circle (1pt)  node[anchor=east] {$P_5$};
\filldraw \TA circle (1pt)  node[anchor=north] {$X_1$};
\filldraw \TB circle (1pt)  node[anchor=south] {$X_2$};
\filldraw \TC circle (1pt)  node[anchor=south] {$X_3$};
\filldraw \TD circle (1pt)  node[anchor=north] {$X_4$};
\filldraw \TE circle (1pt)  node[anchor=north] {$X_5$};
\end{tikzpicture}
\caption{A $T_5$ configuration. The lines drawn are rank-one lines}
\end{figure}
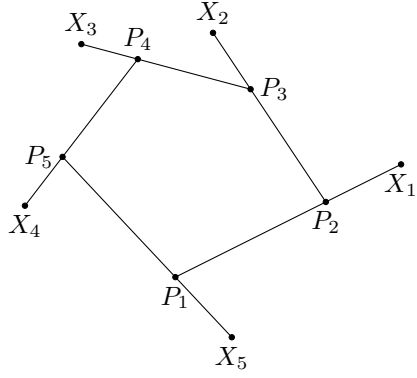

\medskip

We refer to the points
\begin{equation}  \label{eq:inner_points_T4}
P_i := P + \sum_{j=1}^{i-1} C_j, \quad i \in \{1, \ldots,  N\}
\end{equation}
(with $P_1 = P$)
as the inner points of the $T_N$ configuration. 

For the convenience of the reader we recall that a  fundamental property of a $T_N$ configuration $\underline X$ is that the inner points $P_i$  belong to the rank-one convex hull of $\{ X_1, \ldots, X_N\}$.
  \begin{lemma}  \label{l:TNrc}
Assume that $(X_1, \ldots, X_N)$ is a $T_N$ configuration. Then the inner points $P_i$, $i=1,\dots,N$ given by
  \eqref{eq:inner_points_T4}, as well as the line segments $[P_i,X_i]$ are contained in the rank-one convex 
hull $\{X_1, \dots,X_N\}^{rc}$.
\end{lemma}

\begin{proof} 
Otherwise there exists a rank-one convex function $g:  \R^{2 \times 2} \to \R$ such that 
$g(X_i) \le 0$  and $M:= \max_i g(P_i) > 0$. Let $j$ be such that $g(P_j) = M.$
The point  $P_j$ lies in the interior of the line segment $[P_{j-1}, X_{j-1}]$  (here we count $j$ modulo $N$, i.e.,
  we set $P_0 = P_N$ and $X_0 = X_N$).
 Moreover $X_{j-1} - P_{j-1} = \kappa_{j-1} C_{j-1}$ is a rank-one matrix. Thus $g$ is convex
 on this segment. Since $g(P_{j-1}) \le M$,  $g(X_{j-1}) \le 0$ and $M > 0$  it follows that $g(P_j) < M$. This contradicts the 
 assumption $M = g(P_j)$. 
 Thus $P_i \in \{X_1, \dots, X_N\}^{rc}.$ 
  Moreover $\rank(X_i-P_i) = \rank{\kappa_i C_i} = 1$. Thus $[P_i, X_i] \subset  \{X_1, \ldots, X_N\}^{rc}$.

\end{proof}

In general it is not easy to verify  whether a given $N$-tuple of matrices
forms a $T_N$ configuration. For matrices in $\R^{2 \times 2}$,   
the  
 third author identified a
criterion which we now recall. Thus, let $(X_1,\dots,X_N)$ be an ordered set of $2\times 2$ matrices. We set 
\begin{equation}\label{T5e:Amu}	
A^{\mu}_{ij}=
\begin{cases}\det(X_i-X_j)&i<j,\\ 
0&i=j,\\
\mu\det(X_i-X_j)&i>j.\end{cases} 
\end{equation}

\begin{proposition}[\cite{szekelyhidi05}, Prop.\ 2]    \label{pr:criterion_TN}
Let $X_1, \ldots,  X_N \in  \R^{2 \times 2}$    with $\det (X_i - X_j) \ne 0$ for $i \ne j$ and define $A^\mu$ as in \eqref{T5e:Amu}. 
Then $(X_1, \ldots, X_N)$ is a $T_N$ configuration if and only if there exist $\lambda_1, \ldots, \lambda_N > 0$
and $\mu > 1$ such that 
\begin{equation}  \label{eq:Amu_lambda}
A^\mu \lambda = 0.
\end{equation}
Moreover, if  \eqref{eq:Amu_lambda} holds with 
$\lambda_1, \ldots, \lambda_N> 0$
and $\mu > 1$ then the inner points $P_k$ of the $T_N$-configuration  can be chosen as
\begin{equation} \label{eq:Pk}
P_k = \sum_{i=1}^N  \xi^{(k)}_i X_i
\end{equation}
where
\begin{equation} \label{eq:xik}
\xi^{(k)} = \frac1{c_k} v^{(k)} \quad \text{with $ c_k = \sum_{i} v_i^{(k)}$}
\end{equation}
and
\begin{equation} \label{eq:vk}
v^{(k)}_i = \begin{cases}\lambda_i & i\geq k\\ \mu\lambda_i & i< k.\end{cases}
\end{equation}

Conversely, if $(X_1, \ldots, X_N)$ is a $T_N$ configuration with inner points $P_i$, then there exist $\lambda_i > 0$ and $\mu > 1$ such that the 
points $P_i$ can be written in the form  \eqref{eq:Pk}--\eqref{eq:vk} and $A^\mu \lambda = 0$.
\end{proposition}

 In the setting or Proposition~\ref{pr:criterion_TN} we also have
\begin{equation}  \label{eq:det_Pk}
\det P_k = \sum_{i=1}^N  \xi^{(k)}_i \det X_i,
\end{equation}
see \cite{szekelyhidi05}, Lemma 3 and equation (7).

\subsection{$T_4$ configurations in $L\cap\Sigma$}\label{sec:T4ex}
As a first application of Proposition \ref{pr:criterion_TN} we prove Lemma \ref{lem:T4}. 

\begin{proof}[Proof of Lemma \ref{lem:T4}] 
Recall that we define $(X_1,\dots,X_4)$ as
\begin{equation}\label{e:T4example}
\begin{split}
	X_1&=\begin{pmatrix}c&0\\0&1/c\end{pmatrix}, X_2=\begin{pmatrix}1/c&0\\0&c\end{pmatrix},\\ X_3&=\begin{pmatrix}0&-c\\1/c&0\end{pmatrix}, X_4=\begin{pmatrix}0&-1/c\\c&0\end{pmatrix},
	\end{split}
\end{equation}
where $c>0$, and our aim is to show that for certain values of $c$ the set $(X_1,\dots,X_4)$ is a $T_4$-configuration. The matrix $A^{\mu}$ in  \eqref{T5e:Amu} is then given by 
\begin{equation} \label{eq:Amu_X}
A^\mu = \begin{pmatrix} 
0 &    -a &  2   & 2\\
-a\mu  & 0 & 2 & 2 \\
2\mu  & 2\mu  & 0 & -a \\
2\mu  & 2\mu  & -a\mu  & 0
\end{pmatrix},
\end{equation}
where $a=(c-\frac{1}{c})^2$. 
In view of Proposition 5.5 we need to show that there exist $\mu > 1$ and $ \lambda \in (0,\infty)^4$   with
$$ A^\mu \lambda = 0.$$
For $\mu \ne 0$ and $\lambda_1\neq 0$ the equation $A^\mu \lambda = 0$  is equivalent to
\begin{equation} \label{eq:lambda_reduced}
\lambda^T =(1, \mu , \alpha, \mu \alpha) \quad
\text{and} \quad
\begin{pmatrix} -a \mu &  2 (1 + \mu) \\
2 \mu(1 + \mu) & -a\mu
\end{pmatrix}
\binom{1}{\alpha} = 0
\end{equation}
and we are looking for  solutions with $\mu >1$ and $\alpha > 0$. For $\mu \ne 0$ the 
equation for $\binom{1}{\alpha}$ has a non-trivial solution if and only if 
\begin{equation}  \label{eq:det_zero}
0 = \mu^{-1} \det \begin{pmatrix} -a \mu &  2 (1 + \mu) \\
 2\mu (1 + \mu) &- \mu a
\end{pmatrix} = a^2 \mu - 4 (1+ \mu)^2.
\end{equation}
Set 
$$ \gamma = \frac{a^2}{4}.$$
Then  \eqref{eq:det_zero}
is equivalent to 
\begin{equation}   \label{eq:quadratic_eqn_mu}  (1 + \mu)^2 - \gamma \mu = 0.
\end{equation}
The solutions of this equation are given by 
$$ \mu = \frac{\gamma-2}{2} \pm \frac12 \sqrt{(\gamma-2)^2 - 4}.$$
Since $\gamma \ge 0$ a real solutions exist only if $\gamma \ge 4$ and solution with $\mu > 1$ exist 
only if $\gamma > 4$. In fact, for $\gamma > 4$ there is a unique solution with $\mu > 1$, namely, 
\begin{equation}  \label{eq:barmu}
\mu = \bar \mu:= \frac{\gamma-2}{2} +  \frac12 \sqrt{(\gamma-2)^2 - 4}.
\end{equation}
Finally we need to check that  for  $\gamma > 4$ and $\mu = \bar \mu$ there exists a solution with $\alpha > 0$. But this
 follows easily from the fact that $a > 0$.  
  
 Summarizing, we have shown that  for $c > 0$ the matrices $(X_1,\dots,X_4)$ in \eqref{e:T4example} form a $T_4$ configuration
  if and only if $a^2 > 16$ or, equivalently $a > 4$. This is in turn  equivalent to $|c-\frac{1}{c}|>2$. Assuming that $c>1$, this is equivalent to $c^2-2c-1>0$, leading to $c>1+\sqrt{2}$.   
\end{proof}

\subsection{$T_5$-configurations and large $T_5$ sets}

From now on we specialize to the case $N=5$. The following lemma, based on Lemma 2.4 in \cite{szekelyhidi_forster18} gives a simple criterion for the existence of $\mu>1$ in Proposition \ref{pr:criterion_TN}:

\begin{lemma}\label{T5l:mu}
Let $(X_1,\dots,X_5)$ be an ordered set of $2\times 2$ matrices
  with $\det (X_i - X_j) \ne 0$ for $i \ne j$ and let $A^\mu$ be defined as above in \eqref{T5e:Amu}. Further, set
\begin{equation}\label{T5e:alphabeta}
\alpha=A_{12} A_{23} A_{34} A_{45} A_{51},\quad \beta=-\frac{\det A}{2\alpha}.   
\end{equation}
Then, there exists $\mu_*>1$ with $\det A^{\mu_*}=0$ if and only if $\beta>0$ and in this case
\begin{equation}\label{T5e:mu*}
\mu_*=1+\frac{\beta}{2}+\frac{1}{2}\sqrt{\beta^2+4\beta}.
\end{equation}
 Moreover, if $\beta > 0$ then $\ker A^{\mu^*}$ is one-dimensional and thus the vectors $\xi^{(k)}$ in Proposition~\ref{pr:criterion_TN}
are uniquely determined.\\
I am not sure whether this is really needed, but it certainly helps in the discussion about permutations and the identification of the $\xi^{\sigma, k}$.
\end{lemma}

 Note that $\alpha \ne 0$ since $a_{ij} = \det(X_i-X_j) \ne 0$ for $i \ne j$.
 
\begin{proof}
Observe, first of all, that $\mu\mapsto p(\mu):=\det A^{\mu}$ is a degree 4 polynomial, with a trivial zero at $\mu=0$. Moreover, since $(A^\mu)^T=\mu A^{(\mu^{-1})}$, we have the identity $\det A^\mu=\mu^5\det A^{(\mu^{-1})}$. Consequently, $p(-1) = 0$ and thus
$p(\mu) = \mu (\mu+1) (a \mu^2 + b \mu + c)$. The identity $\det A^\mu=\mu^5\det A^{(\mu^{-1})}$ implies that $c=a$.
We claim that 
$$ a = \lim_{\mu \to 0} \frac{p(\mu)}{\mu} = \alpha$$
Indeed the limit can be evaluated by dividing the first colum of $A^\mu$ by $\mu$, setting $\mu = 0$ and evaluating the remaining determinant by 
inspection
\begin{eqnarray*}
a &=& \det\begin{pmatrix} 0& A_{12}& A_{13}& A_{14}& A_{15}\\ A_{12}&0& A_{23}& A_{24}& A_{25}\\
A_{13}&0&0& A_{34}& A_{35}\\ A_{14}&0&0&0& A_{45}\\ A_{15}&0&0&0&0\end{pmatrix}=
\det\begin{pmatrix}  A_{12}& A_{13}& A_{14}& A_{15}&0\\0& A_{23}& A_{24}& A_{25}& A_{12}\\
0&0&A_{34}& A_{35}& A_{13}\\0&0&0& A_{45}& A_{14}\\0&0&0&0& A_{15}\end{pmatrix}  \\ 
&=&  \alpha.
\end{eqnarray*}
Now  the polynomial $\alpha \mu^2 + b \mu + \alpha$  has a zero $\mu_* > 1$ if and only if $\alpha \ne 0$ and $\frac{b}{\alpha} < -2$. 
In that case the other zero is given by $1/{\mu^*}$.
We have  $\det A = p(1) = 2 (2\alpha + b)$ and for $\alpha \ne 0$ we get   $-\beta = 2 + \frac{b}{\alpha}$.   
Hence there exists a zero $\mu^* >1$ if and only if $\beta > 0$ and the expression  \eqref{T5e:mu*} follows easily. 
Moreover $\mu^*$ is a simple zero and thus 
$$ 0 \ne \frac{d}{d\mu}|_{\mu= \mu^*} \det A^\mu = \cof A^{\mu^*} \cdot  \frac{d}{d\mu}|_{\mu= \mu^*} A^\mu.$$ 
Hence $\cof A^{\mu^*} \ne 0$ and
 thus $\rank A^{\mu^*}=4$ 
 and $\dim \ker A^{\mu^*} = 1$.
\end{proof}

\bigskip

Note that the property of being a $T_N$ configuration in Definition \ref{de:T4} depends not just on the set $\{X_1,\dots,X_N\}$ but also on the specific ordering $(X_1,\dots,X_N)$. Thus, one may ask whether the same set of $N$ matrices is a $T_N$ configuration for several different orderings. Indeed, as an example one may easily check that the set $\{X_1,\dots,X_4\}$ in Lemma \ref{lem:T4} (see \eqref{e:T4example}) is a $T_4$ configuration for all $6$ possible orderings. 
Next, we consider the effect of different orderings for $5$-element sets. To fix notation, we denote by $S_5$ the group of permutations of $5$ elements. For our purposes it is helpful to keep track of the orderings induced by permutations, so that, for any $\sigma\in S_5$ we write 
$$
\sigma=\bigl[\sigma(1)\,\sigma(2)\,\sigma(3)\,\sigma(4)\,\sigma(5)\bigr].
$$
Let $\sigma\in S_5$. Applying the general criterion in Proposition \ref{pr:criterion_TN} as well as Lemma~\ref{T5l:mu} to the ordered set 
$$
\left(X_{\sigma(1)},X_{\sigma(2)},X_{\sigma(3)},X_{\sigma(4)},X_{\sigma(5)}\right),
$$  we obtain the following result.

\begin{proposition}   \label{pr:T5_permutation} 
(i) The tuple
$X^\sigma := \left(X_{\sigma(1)},X_{\sigma(2)},X_{\sigma(3)},X_{\sigma(4)},X_{\sigma(5)}\right)$ is a $T_5$ configuration if and only if 
there exist $\mu^{\sigma}>1$ and $\lambda_1^{\sigma},\dots,\lambda_5^{\sigma}>0$ such that
$$
A^{\sigma,\mu^{\sigma}}\lambda^{\sigma}=0,
$$
where
$$
A_{ij}^{\sigma,\mu}=\begin{cases}\det(X_i-X_j)&\sigma^{-1}(i)<\sigma^{-1}(j),\\ 
0&i=j,\\
\mu\det(X_i-X_j)&\sigma^{-1}(i)>\sigma^{-1}(j).\end{cases} 
$$
Here  $\sigma^{-1}$ denotes  the inverse permutation of $\sigma$. 

(ii) If $X^\sigma$ is a $T_5$ configuration then the set of inner points is given by
$\{ P^\sigma_1, \ldots P^\sigma_5\}$ where   
\begin{equation} \label{T5e:Pki}
P_k^\sigma = \sum_{i=1}^5  \xi^{(\sigma, k)}_i X_i
\end{equation}
with  $\xi^{(\sigma, k)} = (\sum_{i} v_i^{(\sigma,k)})^{-1} v^{(\sigma,k)}$, 
\begin{equation}\label{T5e:vki} 
v^{(\sigma,k)}_i = \begin{cases}\lambda^\sigma_i & \sigma^{-1}(i)\geq \sigma^{-1}(k)\\ \mu^\sigma\lambda^\sigma_i & \sigma^{-1}(i)< \sigma^{-1}(k), \end{cases}
\end{equation}
and $v^{(\sigma, k)} = \sum_{i=1}^5 v^{(\sigma, k)}_i$.
In particular $X_k$ is rank-one connected to $P^\sigma_k$. 

(iii) Set $A = A^{\sigma,1}$ and 
\begin{equation}\label{e:permuteddeterminants}
\alpha^\sigma := A_{\sigma(1) \sigma(2)}  A_{\sigma(2) \sigma(3)} A_{\sigma(3) \sigma(4)} A_{\sigma(4) \sigma(5)}  A_{\sigma(5) \sigma(1)}.
\end{equation}
Then $ \mu \mapsto \det A^{\sigma, \mu}$  has a zero  $\mu^* > 1$ if and only if
$$ \beta^\sigma := -  \frac{\det A}{\alpha^\sigma} > 0.$$
If this condition holds then $\mu^*$ is given by  \eqref{T5e:mu*} with $\beta$ replaced by $\beta^\sigma$.
\end{proposition}

\begin{proof}  (i) Fix a permutation $\sigma \in S_5$. For a tuple 
$Y = (Y_1, \ldots, Y_5)$ define $A^\mu(Y)$ as in \eqref{T5e:Amu} with $X_i$ replaced by $Y_i$. 
For $\lambda \in \R^5$ define $\lambda^\sigma$ by $\lambda^\sigma_i = \lambda_{\sigma^{-1}(i)}$. 
It follows directly from the definitions that $A^\mu(X^\sigma)_{ij} = A^{\sigma, \mu}_{\sigma(i) \sigma(j)}$.
In particular
$$ A^\mu(X^\sigma) \lambda = 0 \quad \Longleftrightarrow \quad A^{\sigma, \mu} \lambda^\sigma = 0.$$
Thus assertion (i) follows from   Proposition~\ref{pr:criterion_TN} applied to $X^\sigma$ instead of $X$. 

(ii) Let $\mu^\sigma >1$ and $\lambda \in (0,\infty)^5$ be such that $A^{\mu^\sigma}(X^\sigma) \lambda = 0$.
Define $\xi^{(k)}$ and $P_k$ as in Proposition~\ref{pr:criterion_TN}. Then $X^\sigma$ is a $T_5$ configuration with interior points
$P_k = \sum_{i=1}^5 \xi^{(k)}_i X_{\sigma(i)}$. Moreover $P_k$ is rank-one connected to $(X^\sigma)_k = X_{\sigma(k)}$. 
Unwinding definitions, we see that $\xi^{(k)}_i = \xi^{(\sigma, \sigma(k))}_{\sigma(i)}$.
Thus $P_k = P^{\sigma}_{\sigma(k)}$. Hence $X_{\sigma(k)}$ is rank-one connected to $P^\sigma_{\sigma(k)}$ for all $k$. 

(iii) Set $A(X^\sigma) = A^1(X^\sigma)$. Note that $ A = A^{\sigma, 1}$ is independent of $\sigma$ and $A_{\sigma(i) \sigma(j)} = A(X^\sigma)_{ij}$.
Since $A(X^\sigma)$ is obtained from $A$ by permutation rows and columns with the same permutation we have
 $\det A = \det A(X^\sigma)$. The assertion now follows from Lemma~\ref{T5l:mu} applied to $X^\sigma$. 
\end{proof}

Next, we recall the following definition from \cite{szekelyhidi_forster18}:

\begin{definition}\label{T5d:largeT5}
 We call a five-point set $\{X_1,\ldots,X_5\} \subset \R^{2 \times 2}$ a \emph{large $T_5$-set} if there exist at least three permutations $\sigma_1,\sigma_2,\sigma_3$
such that $(X_{\sigma_j(1)},\dots, X_{\sigma_j(5)})$ is a $T_5$-configuration for each $j=1,2,3$, and moreover the associated rank-one matrices $\{P_i^{\sigma_j}-X_i:\,j=1,2,3\}$ 
are linearly independent for all $i=1,\dots,5$.
\end{definition}

The significance of this definition is the following 
\begin{theorem}[Theorem 2.8 \cite{szekelyhidi_forster18}]  \label{th:large_T5}
Let $\Sigma=\{X\in\R^{2\times 2}:\det X=1\}$.
	If $K=\{X_1,\dots,X_5\}\subset\Sigma$ is a large $T_5$ set, then $K$ admits an in-approximation relative to $\Sigma$.
\end{theorem}

In view of this result the proof of Proposition \ref{prop:largeT5ex} follows once we show 
that the set $\{X_1,\dots X_5\}$ in \eqref{T5e:ex} is a large $T_5$ set. This is the content of Proposition \ref{T5p:ex} below. 
We will use the following criterion to verify the large $T_5$ property.

\begin{proposition} \label{pr:criterion_rank3}  Assume that $\{X_1,\dots,X_5\}\subset\R^{2\times 2}$ is affine non-degenerate, meaning that the affine subspace of $\R^{2\times 2}$ spanned by these $5$ matrices is $4$-dimensional. Then the large $T_5$ property is 
equivalent to the condition that there exist $\sigma_1,\sigma_2,\sigma_3$ $T_5$-configurations and furthermore 
\begin{equation}\label{T5e:rank}
	\textrm{rank }B^{(k)}=3\textrm{ for all }k=1,\dots,5,
\end{equation}
where
$$
B^{(k)}_{ij}=\begin{cases} \lambda_j^{\sigma_i}& \sigma_i^{-1}(k)<\sigma_i^{-1}(j),\\
0&k=j,\\
\mu^{\sigma_i}\lambda_j^{\sigma_i}&	\sigma_i^{-1}(k)>\sigma_i^{-1}(j).\end{cases}
$$
\end{proposition}

\begin{proof} Fix $k$.    
Using the representation \eqref{T5e:Pki} we have
$$
P_k^\sigma-X_k=\sum_{j=1}^5  \xi^{(\sigma, k)}_i (X_j-X_k).$$
Since the  four matrices $X_j - X_k$ for $j \in \{ 1, \ldots,  5\} \setminus \{k\}$ are linearly independent,
the condition of linear independence is equivalent to the condition that the rank of the $3 \times 4$ matrix with entries
$\xi^{(\sigma_i, k)}_j$ is $3$. The assertion follows by  multiplying the $i$-th row of this matrix by $v^{\sigma_i,k} > 0$.
\end{proof}

\subsection{Proof of Proposition~\ref{prop:largeT5ex}}   \label{sec:largeT5ex}
Recall that we look at the matrices 
\begin{equation*}
\begin{split}
X_1&=\begin{pmatrix} c&0\\0&1/c\end{pmatrix},\, 
X_2=\begin{pmatrix} 1/c&0\\0&c\end{pmatrix},\\
X_3&=\begin{pmatrix} 0&-c\\1/c&0\end{pmatrix},\,
X_4=\begin{pmatrix} 0 &-1/c\\ c & 0\end{pmatrix},\,
X_5=\begin{pmatrix} 1&0\\0&1\end{pmatrix}.	
\end{split}
\end{equation*}
with $c > 1$. 

Proposition~\ref{prop:largeT5ex} follows from Theorem~\ref{th:large_T5} and the following result.

\begin{proposition}\label{T5p:ex}
Let  $a=(c-\tfrac1c)^2>0$,  $b=c+\tfrac1c-2>0$ and
\begin{equation*}
\sigma_1=[1\, 2\, 3\, 5\, 4],\quad 	\sigma_2=[1\, 2\, 4\, 5\, 3],\quad 	\sigma_3=[1\, 2\, 5\, 3\, 4].
\end{equation*}
If $ab>8$, then the set $\{X_1,\dots,X_5\}$ is a large $T_5$ set. More precisely, in this case the permutations $\sigma_j$, $j=1,2,3$ correspond
to $T_5$-configurations and the associated rank-one directions $\{P_i^{\sigma_j}-X_i:\,j=1,2,3\}$ are linearly independent for each $i=1,\dots,5$.
\end{proposition}
Elementary calculations show that $ab>8$ holds for instance if $c\geq 3$.

\begin{proof}[Proof of Proposition \ref{T5p:ex}]
Let $X^\sigma = (X_{\sigma(1)}, \ldots, X_{\sigma(5)})$. 
We first show that $X^{\sigma_1}$, $X^{\sigma_2}$ and $X^{\sigma_3}$ are
$T_5$ configurations by using the criterion in Proposition~\ref{pr:T5_permutation}.

To compute $A = A^{\sigma, 1}$ we observe that 
 $\det X_i=1$ for all $i$. 
  Thus
\begin{equation*}
A=\begin{pmatrix}
	0&-a&2&2&-b\\ -a&0&2&2&-b\\2&2&0&-a&2\\2&2&-a&0&2\\-b&-b&2&2&0
\end{pmatrix}	
\end{equation*}
with $a=(c-\tfrac1c)^2>0$, and $b=c+\tfrac1c-2>0$. 
A direct calculation gives
$$
\det A=2a^2(ab^2+4a-16b).
$$
Since $a=b(c^{\sfrac12}+c^{-\sfrac12})^2> 4b$, we see that $\det A>0$. 
Moreover, using the notation from \eqref{e:permuteddeterminants},
\begin{eqnarray*} \alpha^{\sigma_1} &=& A_{12} A_{23} A_{35} A_{54} A_{41} = - 16 a, \\
 \alpha^{\sigma_2} &=& A_{12} A_{24} A_{45} A_{53} A_{31} = - 16 a, \\
  \alpha^{\sigma_3} &=& A_{12} A_{25} A_{53} A_{34} A_{41} = - 4 a^2 b.
  \end{eqnarray*}
Hence $\beta^{\sigma_i}>0$ and therefore the $\mu^{(i)}$ defined by formula \eqref{T5e:mu*}  with $\beta$ replaced by
$\beta^{\sigma_i}$ satisfy $\mu^{(i)}>1$. 
Note also that, since $ab>8$, $\beta^{\sigma_1}=\beta^{\sigma_2}>\beta^{\sigma_3}$ 
and consequently $\mu^{(1)}=\mu^{(2)}>\mu^{(3)}$. In  the following, let us denote
$$
\eta=\mu^{(1)}=\mu^{(2)},\quad \nu=\mu^{(3)}.
$$

It remains to check that the kernels of the matrices $A^{\sigma_i,\mu^{(i)}}$ intersect the 1st octant $(0,\infty)^5$.
 Let us first consider the permutation $\sigma_1=[1 2 3 5 4]$. Then 
$$
A^{\sigma_1,\eta}=\begin{pmatrix}
	0&-a&2&2&-b\\ -a\eta &0&2&2&-b \\ 2\eta &2\eta &0&-a&2\\ 2\eta &2\eta &-a\eta &0&2\eta\\-b\eta &-b\eta &2\eta &2 &0
\end{pmatrix}.	
$$ 
Subtracting the first row from the second and subsequently subtracting appropriate multiples of the second row from the others, we obtain
$$
\tilde A^{\sigma_1,\eta}=\begin{pmatrix}
	0&-a&2&2&-b\\ 
	\eta &-1&0&0&0\\ 
	2\eta(1+\eta) &0&0&-a&2\\ 
	2\eta(1+\eta) &0&-a\eta &0&2\eta\\
	-b\eta(1+\eta) &0 &2\eta &2 &0
\end{pmatrix},
$$ 
and a further row reduction results in the matrix
$$
  \tilde{ \tilde A}^{\sigma_1, \eta}  = \begin{pmatrix}
	0&-a&2&2&-b\\ 
	\eta &-1&0&0&0\\ 
	\lambda^{(1)}_3 &0&-1&0&0\\ 
	\lambda^{(1)}_4 &0&0&-1&0\\
	\lambda^{(1)}_5 &0 &0&0&-1
\end{pmatrix},
$$ 
where $\lambda^{(1)}\in \R^5$ is given by
$$
\lambda^{(1)}=\begin{pmatrix}1&\eta &\tfrac{2}{a}\left((\tfrac{ab}{4}-1)\eta+1\right)&\tfrac{2\eta}{a}\left(\tfrac{ab}{4}-1+\eta\right)&\left(\frac{ab}{4}-2\right)\eta\end{pmatrix}^T
$$ 
In particular, $( \tilde {\tilde A}^{\sigma_1,\eta}\lambda^{(1)})_i=0$ for $i=2,\dots,5$.

Clearly the last four rows of $\tilde {\tilde A}^{\sigma_1,\eta}$
are linearly independent. Since $\det \tilde {\tilde A}^{\sigma_1,\eta} = 0$
it follows that the first row of $\tilde {\tilde A}^{\sigma_1,\eta}$ is a linear combination
of the last four rows. Thus $\tilde {\tilde A}^{\sigma_1,\eta} \lambda^{(1)} = 0$
and hence $A^{\sigma_1, \eta} \lambda^{(1)} = 0$. Since $\rank A^{\sigma_1, \eta} = 4$ it follows that $\lambda^{(1)}$ is the vector which generates the kernel of 
$ A^{\sigma_1, \eta} $. 
Since $a,b>0$, $\eta>1$ and $ab>8$, we see $\lambda^{(1)}_i>0$ for all $i$, so that $\sigma_1$ indeed corresponds to a $T_5$-configuration.

\smallskip

Concerning the case $\sigma_2=[1 2 4 5 3]$ we note that
$$A^{\sigma_2,\eta}=\begin{pmatrix}
	0&-a&2&2&-b\\ -a\eta &0&2&2&-b \\ 2\eta &2\eta &0&-a \eta &2 \eta \\ 2\eta &2\eta &-a &0&2\\-b\eta &-b\eta &2 &2 \eta &0
\end{pmatrix}.
$$
Comparing this  expression with 
$$ A^{\sigma_1,\eta}=\begin{pmatrix}
	0&-a&2&2&-b\\ -a\eta &0&2&2&-b \\ 2\eta &2\eta &0&-a&2\\ 2\eta &2\eta &-a\eta &0&2\eta\\-b\eta &-b\eta &2\eta &2 &0
\end{pmatrix}.	
$$
we see that $A^{\sigma_2, \eta}$ is obtained from $A^{\sigma_1, \eta}$ for swapping 
the 3rd and 4th row and the 3rd and 4th column.

  Hence
$$
\lambda^{(2)}=\begin{pmatrix}1&\eta &\tfrac{2\eta}{a}\left(\tfrac{ab}{4}-1+\eta\right)&\tfrac{2}{a}\left((\tfrac{ab}{4}-1)\eta+1\right)&\left(\frac{ab}{4}-2\right)\eta\end{pmatrix}^T
$$
is the vector generating the $1$-dimensional kernel of $A^{\sigma_2,\eta}$, and, as above, we see that $\lambda^{(2)}_i>0$ for all $i$ under the conditions of the proposition.  

\smallskip

Finally, let us look at $\sigma_3=[1 2 5 3 4]$. Here
$$
A^{\sigma_3,\nu}=\begin{pmatrix}
	0&-a&2&2&-b\\ -a\nu &0&2&2&-b \\ 2\nu &2\nu &0&-a&2\nu\\ 2\nu &2\nu &-a\nu &0&2\nu\\-b\nu &-b\nu &2 &2 &0
\end{pmatrix}.	
$$ 
Proceeding with row-reduction as above, we obtain first
$$
\tilde A^{\sigma_3,\nu}=\begin{pmatrix}
	0&-a&2&2&-b\\ 
	\nu &-1&0&0&0\\ 
	2\nu(1+\nu) &0&0&-a&2\nu\\ 
	2\nu(1+\nu) &0&-a\nu &0&2\nu\\
	-b\nu(1+\nu) &0 &2 &2 &0
\end{pmatrix},
$$ 
and a further row reduction results in the matrix
$$
  \tilde{\tilde A}^{\sigma_3, \nu} 
=
\begin{pmatrix}   	0&-a&2&2&-b\\ 
	\nu &-1&0&0&0\\ 
	\lambda^{(3)}_3 &0&-1&0&0\\ 
	\lambda^{(3)}_4 &0&0&-1&0\\
	\lambda^{(3)}_5 &0 &0&0&-1
\end{pmatrix},
$$ 
where
$$
\lambda^{(3)}=\begin{pmatrix}1&\nu &\tfrac{b}{2}\nu &\tfrac{b}{2}\nu^2&\left(\frac{ab}{4}-1\right)\nu-1\end{pmatrix}^T.
$$

Arguing as before, we deduce that $A^{\sigma_3,\nu}\lambda^{(3)}=0$, and furthermore $\lambda^{(3)}_i>0$ for all $i$ since $a,b>0$, $\nu>1$ and $ab>8$. Therefore $\sigma_3$ also corresponds to a $T_5$ configuration.

\bigskip

In view of Proposition~\ref{pr:criterion_rank3} it only
  remains to check the rank condition in \eqref{T5e:rank}. To this end it suffices to show that for each of the $3\times 5$ matrices $B^{(k)}$ there exists a non-vanishing $3\times 3$ subdeterminant. A judicious choice of the relevant columns in each $B^{(k)}$ leads one to look at 
$$
B^{(1)}_{[2 4 5]},\,B^{(2)}_{[1 4 5]},\,
 B^{(3)}_{[124]},\,
B^{(4)}_{[1 2 3]},\,B^{(5)}_{[1 2 3]},
$$
where $B^{(k)}_{[l m n]}$ denotes the $3\times 3$ matrix formed by restricting $B^{(k)}$ to columns $[l m n ]$. 
Elementary calculations lead to
\begin{align*}
\det B^{(1)}_{[2 4 5]} 
&
 = \det \begin{pmatrix} \lambda^{(1)}_2   &  \lambda^{(1)}_4 & \lambda^{(1)}_5 \\
  \lambda^{(2)}_2 &  \lambda^{(2)}_4 & \lambda^{(2)}_5\\
   \lambda^{(3)}_2   &  \lambda^{(3)}_4 & \lambda^{(3)}_5
\end{pmatrix}
\\
 &= \det \begin{pmatrix}
	\eta &\tfrac{2\eta}{a}\left(\tfrac{ab}{4}-1+\eta\right)&\left(\frac{ab}{4}-2\right)\eta\\
    \eta &\tfrac{2}{a}\left((\tfrac{ab}{4}-1)\eta+1\right)&\left(\frac{ab}{4}-2\right)\eta\\
    \nu &\tfrac{b}{2}\nu^2&\left(\frac{ab}{4}-1\right)\nu-1
\end{pmatrix}\\
&=\det\begin{pmatrix}
	\eta &\tfrac{2\eta}{a}\left(\tfrac{ab}{4}-1+\eta\right)&\left(\frac{ab}{4}-2\right)\eta\\
    0 &\tfrac{2}{a}\left(1-\eta^2\right)&0\\
    \nu &\tfrac{b}{2}\nu^2&\left(\frac{ab}{4}-1\right)\nu-1
\end{pmatrix}\\
&=-\tfrac{2}{a}(\eta^2-1)\eta(\nu-1),\\
  B^{(2)}_{[1 4 5]} 
 & =
 \begin{pmatrix} \eta \lambda^{(1)}_1 &  \lambda^{(1)}_4 & \lambda^{(1)}_5 \\
\eta  \lambda^{(2)}_1 &  \lambda^{(2)}_4 & \lambda^{(2)}_5\\
 \nu   \lambda^{(3)}_1 &   \lambda^{(3)}_4 & \lambda^{(3)}_5
\end{pmatrix}  
=  B^{(1)}_{[245]},  \quad  \text{and thus}  \\
\det   B^{(2)}_{[1 4 5]} & = \det  B^{(1)}_{[245]}.
 \end{align*}

 Moreover, 
 \begin{align*}
\det B^{(3)}_{[124]} &=
\det  \begin{pmatrix} \eta \lambda^{(1)}_1  & \eta   \lambda^{(1)}_2 &   \lambda^{(1)}_4 &  \\
\eta  \lambda^{(2)}_1 & \eta \lambda^{(2)}_2 &   \eta  \lambda^{(2)}_4 &  \\
 \nu   \lambda^{(3)}_1 & \nu \lambda^{(3)}_2 &   \lambda^{(3)}_4 
\end{pmatrix}
  \\
&= 
\det\begin{pmatrix}
	\eta &\eta^2&\tfrac{2\eta}{a}\left(\tfrac{ab}{4}-1+\eta\right)\\
    \eta &\eta^2&\tfrac{2\eta}{a}\left((\tfrac{ab}{4}-1)\eta+1\right)\\
    \nu &\nu^2&   \tfrac{b}{2}\nu^2
\end{pmatrix}  \\
&= \eta^2 \nu
 \det\begin{pmatrix} 
	1 &\eta&\tfrac{2}{a}\left(\tfrac{ab}{4}-1+\eta\right)\\
    1 &\eta&\tfrac{2}{a}\left((\tfrac{ab}{4}-1)\eta+1\right)\\
    1 &\nu&  \tfrac{b}{2}\nu
\end{pmatrix}
\\
&=
\eta^2 \nu
 \det\begin{pmatrix} 
	1 &\eta&\tfrac{2}{a}\left(\tfrac{ab}{4}-1+\eta\right)\\
    0 & 0 &\tfrac{2}{a}     \left((\tfrac{ab}{4}-1)\eta+1     - \tfrac{ab}{4} + 1 - \eta) \right)\\
    1 &\nu&  \tfrac{b}{2}\nu
\end{pmatrix}
\\
&=-\frac{2}{a}\eta^2\nu(\tfrac{ab}{4}-2)(\eta-1)(\nu-\eta)
\end{align*}

This calculation shows in particular that the determinant of three times three matrix  $B^{(3)}_{[124]}$
does not depend on the $33$ entries of the matrix. Using this fact we obtain in the same way
\begin{align*}
\det B^{(4)}_{[1 2 3]}
&=
\det \begin{pmatrix} \eta \lambda^{(1)}_1  & \eta   \lambda^{(1)}_2 &  \eta  \lambda^{(1)}_3 \\
\eta  \lambda^{(2)}_1 & \eta \lambda^{(2)}_2 &    \lambda^{(2)}_3 \\
 \nu   \lambda^{(3)}_1 & \nu \lambda^{(3)}_2 &   \nu \lambda^{(3)}_3 \end{pmatrix}
\\
&=\det\begin{pmatrix}
	\eta &\eta^2&\tfrac{2\eta}{a}\left((\tfrac{ab}{4}-1)\eta+1\right)\\
    \eta &\eta^2&\tfrac{2\eta}{a}\left(\tfrac{ab}{4}-1+\eta\right)\\
    \nu &\nu^2&   \tfrac{b}{2}\nu^2
\end{pmatrix}\\
&=    -\det B^{(3)}_{[124]}  \\%
\det B^{(5)}_{[1 2 3]}  &=
\det  \begin{pmatrix} \eta \lambda^{(1)}_1  & \eta   \lambda^{(1)}_2 &  \eta  \lambda^{(1)}_3  \\
\eta  \lambda^{(2)}_1 & \eta \lambda^{(2)}_2 &    \lambda^{(2)}_3 \\
 \nu   \lambda^{(3)}_1 & \nu \lambda^{(3)}_2 &    \lambda^{(3)}_3 
\end{pmatrix}
\\
&=\det\begin{pmatrix}
	\eta &\eta^2&\tfrac{2\eta}{a}\left((\tfrac{ab}{4}-1)\eta+1\right)\\
    \eta &\eta^2&\tfrac{2\eta}{a}\left(\tfrac{ab}{4}-1+\eta\right)\\
    \nu &\nu^2&     \tfrac{b}{2}\nu
\end{pmatrix}\\
&= \det B^{(3)}_{[124]}. \\
\end{align*}
Since we already know that $\eta>\nu$, $\eta,\nu>1$ and $ab>8$ by assumption, we deduce that none of the $5$ determinants above vanishes, thus showing that the rank condition \eqref{T5e:rank} is satisfied. This concludes the proof.
\end{proof}

\appendix

\section{Proof of Corollary~\ref{cor_heisenberg_example}}~
\label{sec_proof_heisenberg_corollary}

We first recall some notation. The Heisenberg group is 
$$
\H:=\left\{\left[\begin{matrix}1&x_1&x_3\\0&1&x_2\\0&0&1\end{matrix}\right]\mid x_i\in\R\right\}\,.
$$
We let 
$$
X_1:=\left[\begin{matrix}0&1&0\\0&0&0\\0&0&0\end{matrix}\right]\,,\quad X_2:=\left[\begin{matrix}0&0&0\\0&0&1\\0&0&0\end{matrix}\right]\,\quad X_3:=\left[\begin{matrix}0&0&1\\0&0&0\\0&0&0\end{matrix}\right]
$$
be the standard basis for the Lie algebra $\fh$, so $[X_1,X_2]=X_3$ and $[X_1,X_3]=[X_2,X_3]=0$.  We let $\th_1,\th_2,\th_3$ be the dual basis, so $d\th_3=-\th_1\we\th_2$. 
Let $Z(\H):=\exp\R X_3$ be the center of $\H$, and $\pi:\H\ra \H/[\H,\H]$ be the abelianization homomorphism.  
We will identify $\R^3$ with $\H$ by 
$$
(x_1,x_2,x_3)\leftrightarrow \left[\begin{matrix}1&x_1&x_3\\0&1&x_2\\0&0&1\end{matrix}\right]
$$
and the abelianization $\H/[\H,\H]$ with $\R^2$ by $[(x_1,x_2,x_3)]\leftrightarrow (x_1,x_2)\in \R^2$; with these identifications the abelianization homomorphism becomes the projection $\pi(x_1,x_2,x_3)=(x_1,x_2)$.  

We note that in this representation for the Heisenberg group the group action is 
explicitly given by $x \ast y = (x_1+ y_1, x_2 + y_2, x_3 + y_3 + x_1 y_2)$ and the corresponding left-invariant vectorfields and dual forms are given by
\begin{align*}
X_1(x) &= \frac{\partial}{\partial x_1},  \quad  X_2(x) = \frac{\partial}{\partial x_2} + x_1  \frac{\partial}{\partial x_3}, \quad 
 X_3(x) = \frac{\partial}{\partial x_3} \\
\theta_1 &= dx_1, \quad \theta_2 = dx_2, \quad \theta_3 = dx_3 - x_1 dx_2.
\end{align*}

\begin{lemma}
\label{lem_lifting}
Let $f:\R^2\ra\R^2$ be a Lipschitz mapping such that $\det \g f(x)=1$ for a.e. $x\in \R^2$.  Then there exists a mapping $\hat f:\H\ra\H$ such that:
\ben
\item $\hat f$ is a lift of $f$, i.e. $\pi\circ\hat f=f\circ \pi$.
\item $\hat f$ is locally Euclidean Lipschitz, i.e. it defines a locally Lipschitz mapping $(\R^3,d_{\R^3})\ra(\R^3,d_{\R^3})$ under the identification $\R^3\simeq \H$ above. 
\item $\hat f$ preserves the $1$-form $\th_3$, i.e. $\hat f^*\th_3=\th_3$.  
\een
Moreover:
\ben
\setcounter{enumi}{3}
\item  Any mapping $\hat f'$ satisfying (1)-(3) commutes with the action $\H\curvearrowleft Z(\H)$, i.e. $r_g\circ \hat f'=\hat f'\circ r_g$ for every  $g\in Z(\H)$.
\item There is a unique mapping satisfying (1)-(3) up to composition with translation by elements of $Z(\H)$.
\item $\hat f:(\H,d_{CC})\ra (\H,d_{CC})$ is Lipschitz.
\item For a.e. $x\in \R^2$, the map $\hat f$ is differentiable and Pansu differentiable at every point $\hat x\in \pi^{-1}(x)$; moreover the Pansu differential of $\hat f$ is a lift of the differential of $f$, i.e.   $\pi\circ D_P\hat{f}(\hat x)=Df\circ \pi(\hat{x})$.  
\een
\end{lemma}
\begin{proof}

Let $h:\R^2\ra\R^2$ be a smooth map.  Define $\hat h:\H\ra\H$ by
\begin{equation}
\label{eqn_hath_definition}
\hat h(x_1,x_2,x_3):=(h(x_1,x_2),x_3)\,.
\end{equation}
  A calculation gives
\begin{equation}
\label{eqn_hhat_th_3}
\hat h^*\th_3=\th_3+\pi^*\al
\end{equation}
for some smooth $1$-form $\al\in \Om^1(\R^2)$.

For $u:\R^2\ra\R$ smooth we let $S_u:\H\ra\H$ be the vertical shear given by $S_u(x_1,x_2,x_3)=(x_1,x_2,x_3+u(x_1,x_2))$, so $\pi\circ S_u=\pi$, and one gets 
\begin{equation}
\label{eqn_su_th3}
S_u^*\th_3=\th_3+\pi^*du\,.
\end{equation}
Precomposing our initial lift $\hat h$ with the shear $S_u$, we let 
\begin{equation}
\label{eqn_hath1_definition}
\hat h_1:=\hat h\circ S_u\,.
\end{equation}  
Now
\begin{equation}
\label{eqn_h1hat_th3}
\begin{aligned}
\hat h_1^*\th_3&=S_u^*(\hat h^*\th_3)\\
&=\th_3+\pi^*du+S_u^*(\pi^*\al)\\
&=\th_3+\pi^*(du+\al)\,.
\end{aligned}
\end{equation}

Now suppose $h$ is area-preserving, i.e. $h^*(dx_1\we dx_2)=dx_1\we dx_2$.  Then 
\begin{equation}
\label{eqn_dhhat_th3}
\begin{aligned}
d(\hat h^*\th_3)&=\hat h^*(d\th_3)\\
&=-\hat h^*\pi^*(dx_1\we dx_2)\\
&=-(h\circ\pi)^*(dx_1\we dx_2)\\
&=-\pi^*(dx_1\we dx_2)\\
&=d\th_3\,.
\end{aligned}
\end{equation}
Combining \eqref{eqn_hhat_th_3} and \eqref{eqn_dhhat_th3} we get $d\al=0$.  Thus we may choose $u$ such that $du=-\al$, and then \eqref{eqn_h1hat_th3} gives $\hat h_1^*\th_3=\th_3$.  

Taking $h=f$, $\hat f:=\hat{h}_1$ gives assertions (1)-(3) of the lemma when $f$ is smooth.  When $f$ is only Lipschitz the same argument applies, with the caveat that the mappings are locally Euclidean Lipschitz,  the exterior derivative should be interpreted as the distributional exterior derivative, and one has to use the fact that $f^*d\be=df^*\be$ when $f$ is locally Euclidean Lipschitz and both $\be$ and $d\be$ are $L^\infty_{\loc}$.

{\bf(4).}  Suppose $\hat f':\H\ra\H$ satisfies (1)-(3). By (1), for every $(x_1,x_2)\in\R^2$, the map $\hat f'$ takes the fiber $\pi^{-1}((x_1,x_2))$ to the fiber $\pi^{-1}(f(x_1,x_2))$ and, in view of (3), $\hat f'(x_1,x_2,x_3)=(f(x_1,x_2),x_3+w(x_1,x_2))$ for some function $w:\R^2\to\R$. Now, given $g\in Z(\H), x\in \H$ there exists $c\in \R$ such that $g*x=(x_1,x_2,x_3+c)$. Then $g*\hat f(x)=(f(x_1,x_2),x_3+w(x_1,x_2)+c)=\hat f(g*x)$, which gives (4). 

{\bf(5).}  Suppose $\hat f':\H\ra\H$ satisfies (1)-(3).  By (4) we have 
\begin{align*}
\hat f(x_1,x_2,x_3)=(f(x_1,x_2),x_3+w(x_1,x_2))\,,\\
\hat f'(x_1,x_2,x_3)=(f(x_1,x_2),x_3+w'(x_1,x_2))
\end{align*}
for some functions $w,w':\R^2\ra\R$; hence $\hat f'=\hat f\circ  S_u$ for $u=w'-w$.
Note that $u$ must be Lipschitz by (2), so by \eqref{eqn_su_th3} and (3) we have $du=0$ and  therefore $\hat h'(x)= g*\hat h(x)$ for some $g\in Z(\H)$.

{\bf (6) and (7).}  In view of the construction of $\hat f$ (see \eqref{eqn_hath_definition}, \eqref{eqn_hath1_definition}) for a.e. $x\in \R^2$ we get that $\hat f$ is differentiable at every $\hat x\in \pi^{-1}(x)$.  By (1) and (3)  the differential preserves $V_2$ and $V_1$ respectively; moreover the restriction $D\hat f(\hat x)\restr_{V_1}:V_1\ra V_1$ agrees with $Df(x)$ modulo our identification $V_1\simeq\R^2$, and has operator norm $\leq L$, if $f$ is $L$-Lipschitz.  

Let $W_1\subset\H$ be the full measure subset   
$$
\{ \hat{x}\in  \H\mid D\hat{f}(\hat{x})\;\text{is defined and}\;\|D\hat{f}(\hat{x})\restr_{V_1}\|\leq L\}\,.
$$ 
Let $\ga:[0,1]\ra \H$ be a Lipschitz curve, and $\ell_g:\H\ra\H$ denote left translation by $g\in\H$.  It follows from Fubini's theorem that for a full measure subset $W_2\subset \H$, if $g\in  W_2$, then  $\ell_g\circ\ga(t)\in W_1$ for a.e. $t\in [0,1]$; in particular $\ell_g\circ \ga$ is a horizontal curve.   Therefore for $g\in  W_2$, using the chain rule and the length formula for horizontal Lipschitz curves, we have
\begin{align*}
d(\hat{f}\circ\ell_g\circ\ga(0),\hat{f}\circ\ell_g\circ\ga(1))&\leq \int_0^1 \|(\hat{f}\circ\ell_g\circ\ga)'(t)\|\,dt\\
&= \int_0^1 \|(D\hat{f}(\ell_g\circ\ga)(t))(\ga'(t))\|\,dt\\
&\leq \int_0^1L\|\ga'(t)\|dt\\
&=L\cdot\length(\ga)\,.
\end{align*}
Choosing a sequence $g_k\in  W_2$ with $g_k\ra e$ gives 
$$
d(f(\ga(0)),f(\ga(1)))\leq L\cdot\length(\ga).
$$
Since $\ga$ is arbitrary, this gives (6).  

Let $Z$ be the set of points $x\in \R^2$ such that $Df$ is differentiable at $x$ and there exists $\hat x\in \pi^{-1}(x)\subset\H$ such that $\hat f$ is Pansu differentiable at $\hat x$.  By (4) it follows that $\hat f$ is Pansu differentiable at every point in $\pi^{-1}(x)$ when $x\in Z$.  Now (7) follows from the chain rule for Pansu differentials. 
\end{proof}

\begin{proof}[Proof of Corollary~\ref{cor_heisenberg_example}]
Let $f:\R\times\R\ra\R\times\R$ be as in Corollary~\ref{cor_global_nosplit}.  We may apply Lemma~\ref{lem_lifting} to $f$ and $f^{-1}$, obtaining $d_{CC}$-Lipschitz lifts $\hat f,\widehat {f^{-1}}:\H\ra\H$.  By Lemma~\ref{lem_lifting}(5) we may choose the lifts so that $\widehat {f^{-1}}=(\hat f)^{-1}$; hence both $\hat f$ and $(\hat f)^{-1}$ are $d_{CC}$-bi-Lipschitz.  The remaining assertions in Corollary~\ref{cor_heisenberg_example} follow from Lemma~\ref{lem_lifting} and the properties of $f$ in Corollary~\ref{cor_global_nosplit}.  
\end{proof}

\bigskip
\begin{remark}
\label{rem_f1_f2_rigidity}
For $i\in \{1,2\}$ let $\f_i$ be the foliation of $\H$ defined by the left invariant vector field $X_i$, so the leaves of $\f_i$ are left cosets $g\exp\R X_i$ of the $1$-parameter subgroup $\exp\R X_i$.
Suppose $F:\H\ra \H$ is a bi-Lipschitz homeomorphism preserving the foliations $\f_i$ for $i\in \{1,2\}$, i.e. for $i\in \{1,2\}$ and every $g\in \H$, the image of the left coset $g\exp\R X_i$ is a a left coset $g'\exp\R X_i$ for $g'=g'(g)$.  It follows that $F$ arises from a projective transformation, see \cite{kmx_iwasawa}; in particular, if $F$ arises a lift of a bi-Lipschitz homeomorphism $f:\R^2\ra\R^2$ as in Lemma~\ref{lem_lifting}, then $f$ is split and affine.

\end{remark}

\bibliographystyle{amsalpha}
\bibliography{product_quotient}

\end{document}

%% file: macros_kmx.tex
\usepackage{amsmath,amssymb,amsthm,graphicx,color,amscd}
\usepackage[usenames,dvipsnames]{xcolor}
\usepackage{enumitem}

\usepackage{hyperref}

\numberwithin{equation}{section}
\hypersetup{bookmarksdepth=3}

\theoremstyle{plain}
\newtheorem{theorem}[equation]{Theorem}

\newtheorem{corollary}[equation]{Corollary}

\newtheorem{proposition}[equation]{Proposition}
\newtheorem{lemma}[equation]{Lemma}

\theoremstyle{definition}
\newtheorem{definition}[equation]{Definition}
\newtheorem{example}[equation]{Example}

\newtheorem{question}[equation]{Question}

\theoremstyle{remark}
\newtheorem{remark}[equation]{Remark}

\newcommand{\al}{\alpha}

\newcommand{\be}{\beta}
\newcommand{\ben}{\begin{enumerate}}
\newcommand{\bit}{\begin{itemize}}
\newcommand{\C}{\mathbb{C}}

\newcommand{\cof}{\operatorname{cof}}

\newcommand{\dist}{\operatorname{dist}}
\newcommand{\een}{\end{enumerate}}
\newcommand{\eit}{\end{itemize}}

\newcommand{\eps}{\varepsilon}

\newcommand{\fh}{\mathfrak{h}}

\newcommand{\f}{\mathcal{F}}
\newcommand{\ga}{\gamma}

\renewcommand{\H}{\mathbb{H}}

\newcommand{\length}{\operatorname{length}}

\newcommand{\loc}{\operatorname{loc}}
\newcommand{\lra}{\longrightarrow}

\newcommand{\Om}{\Omega}

\newcommand{\R}{\mathbb{R}}
\newcommand{\ra}{\rightarrow}

\newcommand{\rank}{\operatorname{rank}}
\definecolor{gray}{gray}{0.7}

\newcommand{\restr}{\mbox{\Large \(|\)\normalsize}}

\newcommand{\Span}{\operatorname{span}}

\renewcommand{\th}{\theta}

\newcommand{\we}{\wedge}

\def\XXint#1#2#3{{\setbox0=\hbox{$#1{#2#3}{\int}$ }
\vcenter{\hbox{$#2#3$ }}\kern-.6\wd0}}


%% file: product_quotient.bib
@book {ahlfors66,
    AUTHOR = {Ahlfors, Lars V.},
     TITLE = {Lectures on quasiconformal mappings},
    SERIES = {Van Nostrand Mathematical Studies},
    VOLUME = {No. 10},
      NOTE = {Manuscript prepared with the assistance of Clifford J. Earle,
              Jr},
 PUBLISHER = {D. Van Nostrand Co., Inc., Toronto, Ont.-New York-London},
      YEAR = {1966},
     PAGES = {v+146},
   MRCLASS = {30.47},
  MRNUMBER = {200442},
MRREVIEWER = {P.\ Caraman},
}

@incollection {MR4944781,
    AUTHOR = {R{\"u}land, Angkana},
     TITLE = {Microstructures in the modelling of shape-memory alloys:
              rigidity, flexibility and scaling},
 BOOKTITLE = {Variational and {PDE} methods in nonlinear science},
    SERIES = {Lecture Notes in Math.},
    VOLUME = {2366},
     PAGES = {83--144},
 PUBLISHER = {Springer, Cham},
      YEAR = {[2025] \copyright 2025},
      ISBN = {978-3-031-87201-3; 978-3-031-87202-0},
   MRCLASS = {74M25},
  MRNUMBER = {4944781},
       DOI = {10.1007/978-3-031-87202-0\_2},
       URL = {https://doi.org/10.1007/978-3-031-87202-0_2},
}

@book{scheffer74,
	author = {Scheffer, Vladimir},
	mrclass = {Thesis},
	mrnumber = {2624766},
	note = {Thesis (Ph.D.)--Princeton University},
	pages = {116},
	publisher = {ProQuest LLC, Ann Arbor, MI},
	title = {Regularity and irregularity of solutions to nonlinear second--order elliptic systems of partial differential equations and inequalities},
	url = {http://gateway.proquest.com/openurl?url_ver=Z39.88-2004&rft_val_fmt=info:ofi/fmt:kev:mtx:dissertation&res_dat=xri:pqdiss&rft_dat=xri:pqdiss:7506676},
	year = {1974}}

@article {kmsx_counterexample,
    AUTHOR = {Kleiner, Bruce and M\"uller, Stefan and Sz\'ekelyhidi, Jr.,
              L\'aszl\'o{} and Xie, Xiangdong},
     TITLE = {Rigidity of {E}uclidean product structure: breakdown for low
              {S}obolev exponents},
   JOURNAL = {Commun. Pure Appl. Anal.},
  FJOURNAL = {Communications on Pure and Applied Analysis},
    VOLUME = {23},
      YEAR = {2024},
    NUMBER = {10},
     PAGES = {1569--1607},
      ISSN = {1534-0392,1553-5258},
   MRCLASS = {35R70 (30C62 30C65 35A35 46E35)},
  MRNUMBER = {4799456},
       DOI = {10.3934/cpaa.2024029},
       URL = {https://doi.org/10.3934/cpaa.2024029},
}

@incollection {CDS2012,
    AUTHOR = {Conti, Sergio and De Lellis, Camillo and Sz\'ekelyhidi, Jr.,
              L\'aszl\'o},
     TITLE = {{$h$}-principle and rigidity for {$C^{1,\alpha}$} isometric
              embeddings},
 BOOKTITLE = {Nonlinear partial differential equations},
    SERIES = {Abel Symp.},
    VOLUME = {7},
     PAGES = {83--116},
 PUBLISHER = {Springer, Heidelberg},
      YEAR = {2012},
      ISBN = {978-3-642-25360-7; 978-3-642-25361-4},
   MRCLASS = {53C24 (58A07)},
  MRNUMBER = {3289360},
MRREVIEWER = {Toru\ Yoshiyasu},
       DOI = {10.1007/978-3-642-25361-4\_5},
       URL = {https://doi.org/10.1007/978-3-642-25361-4_5},
}

@article{szekelyhidi_forster18,
	author = {F\"{o}rster, Clemens and Sz\'{e}kelyhidi, Jr., L\'{a}szl\'{o}},
	doi = {10.1007/s00526-017-1293-7},
	fjournal = {Calculus of Variations and Partial Differential Equations},
	issn = {0944-2669},
	journal = {Calc. Var. Partial Differential Equations},
	mrclass = {26B25 (49K21)},
	mrnumber = {3740399},
	mrreviewer = {Ernesto Aranda},
	number = {1},
	pages = {Paper No. 19, 12},
	title = {{$T_5$}-configurations and non-rigid sets of matrices},
	url = {https://mathscinet.ams.org/mathscinet-getitem?mr=3740399},
	volume = {57},
	year = {2018}}

@conference{murat77,
	author = {Murat, F.},
	title = {H-convergence, {S}{\'e}minaire d'analyse fonctionelle et num{\'e}rique, {Universit\'e d'Alger}, 1977--78. { \rm English translation: Murat, F. \& Tartar, L.}, {H}-convergence, { \rm Topics in the mathematical modelling of composite materials, 21--43, {P}rogr. {N}onlinear {D}ifferential {E}quations {A}ppl. , 31, {B}irkh{\"a}user, {B}oston, {MA,} 1997. MR 1493039}},
	year = {1977}}

@incollection{tartar93,
	author = {Tartar, Luc},
	booktitle = {Microstructure and phase transition},
	doi = {10.1007/978-1-4613-8360-4_12},
	mrclass = {49N60 (49J10 73C50 73V25)},
	mrnumber = {1320538},
	mrreviewer = {John M. Ball},
	pages = {191--204},
	publisher = {Springer, New York},
	series = {IMA Vol. Math. Appl.},
	title = {Some remarks on separately convex functions},
	url = {https://mathscinet.ams.org/mathscinet-getitem?mr=1320538},
	volume = {54},
	year = {1993}}

@article{casadio_tarabusi93,
	author = {Casadio Tarabusi, Enrico},
	fjournal = {Ricerche di Matematica},
	issn = {0035-5038},
	journal = {Ricerche Mat.},
	mrclass = {49N15 (26B25 73C50 73V25)},
	mrnumber = {1283802},
	mrreviewer = {Jean-Pierre Crouzeix},
	number = {1},
	pages = {11--24},
	title = {An algebraic characterization of quasi-convex functions},
	url = {https://mathscinet.ams.org/mathscinet-getitem?mr=1283802},
	volume = {42},
	year = {1993}}

@article{nesi_milton91,
	author = {Nesi, Vincenzo and Milton, Graeme W.},
	doi = {10.1016/0022-5096(91)90039-Q},
	fjournal = {Journal of the Mechanics and Physics of Solids},
	issn = {0022-5096},
	journal = {J. Mech. Phys. Solids},
	mrclass = {73R05 (73B99 73K20 73V25 78A25 82D25)},
	mrnumber = {1106125},
	mrreviewer = {Philippe Boulanger},
	number = {4},
	pages = {525--542},
	title = {Polycrystalline configurations that maximize electrical resistivity},
	url = {https://mathscinet.ams.org/mathscinet-getitem?mr=1106125},
	volume = {39},
	year = {1991}}

@article{aumann_hart86,
	author = {Aumann, Robert J. and Hart, Sergiu},
	doi = {10.1007/BF02764940},
	fjournal = {Israel Journal of Mathematics},
	issn = {0021-2172},
	journal = {Israel J. Math.},
	mrclass = {90D15 (52A05 90D20)},
	mrnumber = {852476},
	mrreviewer = {Erwin Klein},
	number = {2},
	pages = {159--180},
	title = {Bi-convexity and bi-martingales},
	url = {https://mathscinet.ams.org/mathscinet-getitem?mr=852476},
	volume = {54},
	year = {1986}}

@article{szekelyhidi04,
	author = {Sz{\'{e}}kelyhidi, Jr., L\'{a}szl\'{o}},
	doi = {10.1007/s00205-003-0300-7},
	fjournal = {Archive for Rational Mechanics and Analysis},
	issn = {0003-9527},
	journal = {Arch. Ration. Mech. Anal.},
	mrclass = {49N60 (35J50 49J10 74G40 74G65)},
	mrnumber = {2048569},
	mrreviewer = {Marc Oliver Rieger},
	number = {1},
	pages = {133--152},
	title = {The regularity of critical points of polyconvex functionals},
	url = {https://mathscinet.ams.org/mathscinet-getitem?mr=2048569},
	volume = {172},
	year = {2004}}

@incollection{tartar15,
	author = {Tartar, Luc},
	booktitle = {Differential geometry and continuum mechanics},
	doi = {10.1007/978-3-319-18573-6_1},
	mrclass = {74-01 (53Z05)},
	mrnumber = {3447482},
	pages = {3--26},
	publisher = {Springer, Cham},
	series = {Springer Proc. Math. Stat.},
	title = {Compensated compactness with more geometry},
	url = {https://mathscinet.ams.org/mathscinet-getitem?mr=3447482},
	volume = {137},
	year = {2015}}

@article{murat78,
	author = {Murat, Fran\c{c}ois},
	fjournal = {Annali della Scuola Normale Superiore di Pisa. Classe di Scienze. Serie IV},
	issn = {0391-173X},
	journal = {Ann. Scuola Norm. Sup. Pisa Cl. Sci. (4)},
	mrclass = {46E35 (49A99)},
	mrnumber = {506997},
	mrreviewer = {Pierre-Louis Lions},
	number = {3},
	pages = {489--507},
	title = {Compacit\'{e} par compensation},
	url = {https://mathscinet.ams.org/mathscinet-getitem?mr=506997},
	volume = {5},
	year = {1978}}

@article{astala_faraco_szekelyhidi08,
	author = {Astala, Kari and Faraco, Daniel and Sz\'{e}kelyhidi, Jr., L\'{a}szl\'{o}},
	fjournal = {Annali della Scuola Normale Superiore di Pisa. Classe di Scienze. Serie V},
	issn = {0391-173X},
	journal = {Ann. Sc. Norm. Super. Pisa Cl. Sci. (5)},
	mrclass = {35J25 (30C62 35B65 35D10)},
	mrnumber = {2413671},
	mrreviewer = {Leonid V. Kovalev},
	number = {1},
	pages = {1--50},
	title = {Convex integration and the {$L^p$} theory of elliptic equations},
	url = {https://mathscinet.ams.org/mathscinet-getitem?mr=2413671},
	volume = {7},
	year = {2008}}

@article{sychev01,
	author = {Sychev, M. A.},
	doi = {10.1007/PL00009929},
	fjournal = {Calculus of Variations and Partial Differential Equations},
	issn = {0944-2669},
	journal = {Calc. Var. Partial Differential Equations},
	mrclass = {35R70 (49K24)},
	mrnumber = {1861098},
	mrreviewer = {Tzanko D. Donchev},
	number = {2},
	pages = {213--229},
	title = {Comparing two methods of resolving homogeneous differential inclusions},
	url = {https://mathscinet.ams.org/mathscinet-getitem?mr=1861098},
	volume = {13},
	year = {2001}}

@misc{kirchheim03,
	author = {Kirchheim, Bernd},
	howpublished = {Habilitation thesis, University of Leipzig, https://www.mis.mpg.de/preprints/ln/lecturenote-1603.pdf},
	title = {Analysis and geometry of microstructure},
	year = {2003}}

@book{dacorogna_book08,
	author = {Dacorogna, Bernard},
	edition = {Second},
	isbn = {978-0-387-35779-9},
	mrclass = {49-02 (49J10 49J45 74B20)},
	mrnumber = {2361288},
	mrreviewer = {Pietro Celada},
	pages = {xii+619},
	publisher = {Springer, New York},
	series = {Applied Mathematical Sciences},
	title = {Direct methods in the calculus of variations},
	url = {https://mathscinet.ams.org/mathscinet-getitem?mr=2361288},
	volume = {78},
	year = {2008}}

@incollection{muller_sverak96,
	author = {M\"{u}ller, Stefan and {\v{S}}ver\'{a}k, Vladimir},
	booktitle = {Geometric analysis and the calculus of variations},
	mrclass = {49J10 (58C35 73B30 73V25)},
	mrnumber = {1449410},
	mrreviewer = {U. D'Ambrosio},
	pages = {239--251},
	publisher = {Int. Press, Cambridge, MA},
	title = {Attainment results for the two-well problem by convex integration},
	url = {https://mathscinet.ams.org/mathscinet-getitem?mr=1449410},
	year = {1996}}

@article{fonseca_muller_pedregal98,
	author = {Fonseca, Irene and M\"{u}ller, Stefan and Pedregal, Pablo},
	doi = {10.1137/S0036141096306534},
	fjournal = {SIAM Journal on Mathematical Analysis},
	issn = {0036-1410},
	journal = {SIAM J. Math. Anal.},
	mrclass = {49J45 (35B99)},
	mrnumber = {1617712},
	mrreviewer = {Georg K. Dolzmann},
	number = {3},
	pages = {736--756},
	title = {Analysis of concentration and oscillation effects generated by gradients},
	url = {https://mathscinet.ams.org/mathscinet-getitem?mr=1617712},
	volume = {29},
	year = {1998}}

@article{muller_sverak_constraint,
	author = {M\"{u}ller, S. and {\v{S}}ver\'{a}k, V.},
	doi = {10.1007/s100970050012},
	fjournal = {Journal of the European Mathematical Society (JEMS)},
	issn = {1435-9855},
	journal = {J. Eur. Math. Soc. (JEMS)},
	mrclass = {35J60 (35D10 35J20 49J10 74N05)},
	mrnumber = {1728376},
	mrreviewer = {Jean-Pierre Raymond},
	number = {4},
	pages = {393--422},
	title = {Convex integration with constraints and applications to phase transitions and partial differential equations},
	url = {https://mathscinet.ams.org/mathscinet-getitem?mr=1728376},
	volume = {1},
	year = {1999}}

@article{szekelyhidi05,
	author = {Sz{\'{e}}kelyhidi, Jr., L\'{a}szl\'{o}},
	doi = {10.1007/s00526-004-0272-y},
	fjournal = {Calculus of Variations and Partial Differential Equations},
	issn = {0944-2669},
	journal = {Calc. Var. Partial Differential Equations},
	mrclass = {49J45 (52A30 74G65 74N15)},
	mrnumber = {2118899},
	mrreviewer = {Karim Trabelsi},
	number = {3},
	pages = {253--281},
	title = {Rank-one convex hulls in {$\Bbb R^{2\times 2}$}},
	url = {https://mathscinet.ams.org/mathscinet-getitem?mr=2118899},
	volume = {22},
	year = {2005}}

@article{ball_invertibility81,
	author = {Ball, J. M.},
	doi = {10.1017/S030821050002014X},
	fjournal = {Proceedings of the Royal Society of Edinburgh. Section A. Mathematics},
	issn = {0308-2105},
	journal = {Proc. Roy. Soc. Edinburgh Sect. A},
	mrclass = {73C50 (26B99 58C15)},
	mrnumber = {616782},
	mrreviewer = {R. R. Huilgol},
	number = {3-4},
	pages = {315--328},
	title = {Global invertibility of {S}obolev functions and the interpenetration of matter},
	url = {https://mathscinet.ams.org/mathscinet-getitem?mr=616782},
	volume = {88},
	year = {1981}}

@misc{kmx_approximation_low_p,
	author = {Kleiner, B. and M{\"{u}}ller, S. and Xie, X.},
	year={2021},
    eprint={2007.06694},
    archivePrefix={arXiv},
    primaryClass={math.DG},
    howpublished={https://arxiv.org/abs/2007.06694}, 
	title = {Pansu pullback and exterior differentiation for {S}obolev maps on {C}arnot groups}}

@misc{KMX1new,
      title={Pansu pullback and rigidity of mappings between Carnot groups}, 
      author={Kleiner, B. and M{\"{u}}ller, S. and Xie, X.},
      year={2021},
      eprint={2004.09271},
      archivePrefix={arXiv},
      primaryClass={math.DG},
      howpublished={https://arxiv.org/abs/2004.09271}, 
}

@article {kmx_iwasawa,
    AUTHOR = {Kleiner, B. and M\"uller, S. and Xie, X.},
     TITLE = {Rigidity of flag manifolds},
   JOURNAL = {Indiana Univ. Math. J.},
  FJOURNAL = {Indiana University Mathematics Journal},
    VOLUME = {74},
      YEAR = {2025},
    NUMBER = {1},
     PAGES = {195--224},
      ISSN = {0022-2518,1943-5258},
   MRCLASS = {53C24 (30C65 53C30)},
  MRNUMBER = {4883589},
MRREVIEWER = {Irina\ G.\ Markina},
}

@article{pansu,
	author = {Pansu, P.},
	doi = {10.2307/1971484},
	fjournal = {Annals of Mathematics. Second Series},
	issn = {0003-486X},
	journal = {Ann. of Math. (2)},
	mrclass = {53C20 (22E40)},
	mrnumber = {979599},
	mrreviewer = {Gudlaugur Thorbergsson},
	number = {1},
	pages = {1--60},
	title = {M\'{e}triques de {C}arnot-{C}arath\'{e}odory et quasiisom\'{e}tries des espaces sym\'{e}triques de rang un},
	url = {https://mathscinet.ams.org/mathscinet-getitem?mr=979599},
	volume = {129},
	year = {1989}}

@article{buckmaster_vicol19,
	author = {Buckmaster, T. and Vicol, V.},
	doi = {10.4007/annals.2019.189.1.3},
	fjournal = {Annals of Mathematics. Second Series},
	issn = {0003-486X},
	journal = {Ann. of Math. (2)},
	mrclass = {35Q30 (35Q31 35Q35 76D05 76F02)},
	mrnumber = {3898708},
	mrreviewer = {Isabelle Gruais},
	number = {1},
	pages = {101--144},
	title = {Nonuniqueness of weak solutions to the {N}avier-{S}tokes equation},
	url = {https://doi.org/10.4007/annals.2019.189.1.3},
	volume = {189},
	year = {2019}}

@article{isett18,
	author = {Isett, P.},
	doi = {10.4007/annals.2018.188.3.4},
	fjournal = {Annals of Mathematics. Second Series},
	issn = {0003-486X},
	journal = {Ann. of Math. (2)},
	mrclass = {35Q31 (35A02 35D30 76B03 76F02 76F05)},
	mrnumber = {3866888},
	mrreviewer = {Benedetta Ferrario},
	number = {3},
	pages = {871--963},
	title = {A proof of {O}nsager's conjecture},
	url = {https://doi.org/10.4007/annals.2018.188.3.4},
	volume = {188},
	year = {2018}}

@article{delellis_szekelyhidi16,
	author = {De Lellis, C. and Sz\'{e}kelyhidi, Jr., L.},
	doi = {10.1090/bull/1549},
	fjournal = {American Mathematical Society. Bulletin. New Series},
	issn = {0273-0979},
	journal = {Bull. Amer. Math. Soc. (N.S.)},
	mrclass = {35Q31 (35A01 35D30 53A99 53C21 76F02)},
	mrnumber = {3619726},
	mrreviewer = {Benedetta Ferrario},
	number = {2},
	pages = {247--282},
	title = {High dimensionality and h-principle in {PDE}},
	url = {https://doi.org/10.1090/bull/1549},
	volume = {54},
	year = {2017}}

@article{delellis_szekelyhidi09,
	author = {De Lellis, C. and Sz\'{e}kelyhidi, Jr., L.},
	doi = {10.4007/annals.2009.170.1417},
	fjournal = {Annals of Mathematics. Second Series},
	issn = {0003-486X},
	journal = {Ann. of Math. (2)},
	mrclass = {35Q31 (34A60 35D30 76B03)},
	mrnumber = {2600877},
	mrreviewer = {Fr\'{e}d\'{e}ric Charve},
	number = {3},
	pages = {1417--1436},
	title = {The {E}uler equations as a differential inclusion},
	url = {https://doi.org/10.4007/annals.2009.170.1417},
	volume = {170},
	year = {2009}}

@book{gromov_pdr,
	author = {Gromov, M.},
	doi = {10.1007/978-3-662-02267-2},
	isbn = {3-540-12177-3},
	mrclass = {58G99 (35A99 35B99 53C42 58-02)},
	mrnumber = {864505},
	mrreviewer = {Hung-Hsi Wu},
	pages = {x+363},
	publisher = {Springer-Verlag, Berlin},
	series = {Ergebnisse der Mathematik und ihrer Grenzgebiete (3) [Results in Mathematics and Related Areas (3)]},
	title = {Partial differential relations},
	url = {https://doi.org/10.1007/978-3-662-02267-2},
	volume = {9},
	year = {1986}}

@article{muller_sverak03,
	author = {M\"{u}ller, S. and {\v{S}}ver\'{a}k, V.},
	doi = {10.4007/annals.2003.157.715},
	fjournal = {Annals of Mathematics. Second Series},
	issn = {0003-486X},
	journal = {Ann. of Math. (2)},
	mrclass = {35D10 (35J45 35J50 49J10 49N60)},
	mrnumber = {1983780},
	mrreviewer = {John M. Ball},
	number = {3},
	pages = {715--742},
	title = {Convex integration for {L}ipschitz mappings and counterexamples to regularity},
	url = {https://doi.org/10.4007/annals.2003.157.715},
	volume = {157},
	year = {2003}}

@incollection{muller99,
	author = {M{\"{u}}ller, S.},
	booktitle = {Calculus of variations and geometric evolution problems ({C}etraro, 1996)},
	doi = {10.1007/BFb0092670},
	mrclass = {49J45 (35A15 58E50 74B20 74N15)},
	mrnumber = {1731640},
	mrreviewer = {John M. Ball},
	pages = {85--210},
	publisher = {Springer, Berlin},
	series = {Lecture Notes in Math.},
	title = {Variational models for microstructure and phase transitions},
	url = {https://doi.org/10.1007/BFb0092670},
	volume = {1713},
	year = {1999}}

@book{dacorogna_marcellini99,
	author = {Dacorogna, B. and Marcellini, P.},
	doi = {10.1007/978-1-4612-1562-2},
	isbn = {0-8176-4121-1},
	mrclass = {35A25 (35F20 35G20)},
	mrnumber = {1702252},
	mrreviewer = {Jolanta Przybycin},
	pages = {xiv+273},
	publisher = {Birkh\"{a}user Boston, Inc., Boston, MA},
	series = {Progress in Nonlinear Differential Equations and their Applications},
	title = {Implicit partial differential equations},
	url = {https://doi.org/10.1007/978-1-4612-1562-2},
	volume = {37},
	year = {1999}}

@article{scheffer93,
	author = {Scheffer, V.},
	doi = {10.1007/BF02921318},
	fjournal = {The Journal of Geometric Analysis},
	issn = {1050-6926},
	journal = {J. Geom. Anal.},
	mrclass = {35Q35 (28A80 76C99)},
	mrnumber = {1231007},
	mrreviewer = {Helena J. Nussenzveig Lopes},
	number = {4},
	pages = {343--401},
	title = {An inviscid flow with compact support in space-time},
	url = {https://doi.org/10.1007/BF02921318},
	volume = {3},
	year = {1993}}

@article{murat81,
	author = {Murat, F.},
	fjournal = {Annali della Scuola Normale Superiore di Pisa. Classe di Scienze. Serie IV},
	issn = {0391-173X},
	journal = {Ann. Scuola Norm. Sup. Pisa Cl. Sci. (4)},
	mrclass = {46E35 (35B99)},
	mrnumber = {616901},
	mrreviewer = {John M. Ball},
	number = {1},
	pages = {69--102},
	title = {Compacit\'{e} par compensation: condition n\'{e}cessaire et suffisante de continuit\'{e} faible sous une hypoth\`ese de rang constant},
	url = {http://www.numdam.org/item?id=ASNSP_1981_4_8_1_69_0},
	volume = {8},
	year = {1981}}

@incollection{tartar79,
	author = {Tartar, L.},
	booktitle = {Nonlinear analysis and mechanics: {H}eriot-{W}att {S}ymposium, {V}ol. {IV}},
	mrclass = {35A35 (35Q99 49A50 49D50)},
	mrnumber = {584398},
	mrreviewer = {R. Schumann},
	pages = {136--212},
	publisher = {Pitman, Boston, Mass.-London},
	series = {Res. Notes in Math.},
	title = {Compensated compactness and applications to partial differential equations},
	volume = {39},
	year = {1979}}

@article{nash54,
	author = {Nash, J.},
	doi = {10.2307/1969840},
	fjournal = {Annals of Mathematics. Second Series},
	issn = {0003-486X},
	journal = {Ann. of Math. (2)},
	mrclass = {53.0X},
	mrnumber = {65993},
	mrreviewer = {S. Chern},
	pages = {383--396},
	title = {{$C^1$} isometric imbeddings},
	url = {https://doi.org/10.2307/1969840},
	volume = {60},
	year = {1954}}

@article{Sverak:1993va,
author = {{\v{S}}ver{\'a}k, Vladimir},
title = {{On Tartar's conjecture}},
journal = {Annales de l'Institut Henri Poincare/Analyse non lineaire},
year = {1993},
volume = {10},
number = {4},
pages = {405--412}
}

@article{Lorent:2019kv,
author = {Lorent, Andrew and Peng, Guanying},
title = {{Null Lagrangian measures in subspaces, compensated compactness and conservation laws}},
journal = {Arch. Rational Mech. Anal.},
year = {2019},
volume = {234},
number = {2},
pages = {857--910}
}

@article{Ball:1980fy,
author = {Ball, John M},
title = {{Strict convexity, strong ellipticity, and regularity in the calculus of variations}},
journal = {Math. Proc. Camb. Phil. Soc.},
year = {1980},
volume = {87},
number = {3},
pages = {501--513}
}

@article{Chlebik:2002va,
author = {Chlebik, Miroslav and Kirchheim, Bernd},
title = {{Rigidity for the four gradient problem}},
journal = {J. Reine Angew. Math.},
year = {2002},
volume = {2002},
number = {551},
pages = {1--9}
}

@incollection{Kirchheim:2002wc,
author = {Kirchheim, Bernd and {\v{S}}ver{\'a}k, Vladimir and M{\"u}ller, Stefan},
title = {{Studying nonlinear pde by geometry in matrix space}},
booktitle = {Geometric analysis and nonlinear partial differential equations},
year = {2003},
pages = {347--395},
publisher = {Springer, Berlin}
}

@incollection{SzekelyhidiJr:2014tu,
author = {Sz{\'e}kelyhidi Jr, L{\'a}szl{\'o}},
title = {{From Isometric Embeddings to Turbulence}},
booktitle = {HCDTE Lecture Notes. Part II. Nonlinear Hyperbolic PDEs, Dispersive and Transport Equations},
year = {2014},
pages = {1--66},
publisher = {American Institute of Mathematical Sciences}
}

@article{Buckmaster:2019et,
author = {Buckmaster, Tristan and Vicol, Vlad},
title = {{Convex integration and phenomenologies in turbulence}},
journal = {EMS Surv. Math. Sci.},
year = {2019},
volume = {6},
number = {1},
pages = {173--263}
}

@article {LeonettiNesi,
    AUTHOR = {Leonetti, F. and Nesi, V.},
     TITLE = {Quasiconformal solutions to certain first order systems and
              the proof of a conjecture of {G}. {W}. {M}ilton},
   JOURNAL = {J. Math. Pures Appl. (9)},
  FJOURNAL = {Journal de Math\'ematiques Pures et Appliqu\'ees. Neuvi\`eme
              S\'erie},
    VOLUME = {76},
      YEAR = {1997},
    NUMBER = {2},
     PAGES = {109--124},
      ISSN = {0021-7824},
   MRCLASS = {35J15 (30C62 35F05)},
  MRNUMBER = {1432370},
       DOI = {10.1016/S0021-7824(97)89947-3},
       URL = {https://doi.org/10.1016/S0021-7824(97)89947-3},
}

@article {BojarskiDOnofrioetal,
    AUTHOR = {Bojarski, Bogdan and D'Onofrio, Luigi and Iwaniec, Tadeusz and
              Sbordone, Carlo},
     TITLE = {{$G$}-closed classes of elliptic operators in the complex
              plane},
   JOURNAL = {Ricerche Mat.},
  FJOURNAL = {Ricerche di Matematica},
    VOLUME = {54},
      YEAR = {2005},
    NUMBER = {2},
     PAGES = {403--432},
      ISSN = {0035-5038},
   MRCLASS = {30G20 (30C62 35J55)},
  MRNUMBER = {2289490},
MRREVIEWER = {M.\ Yu.\ Vasil\cprime chik},
}
